\documentclass[11 pt]{amsart}
\usepackage{latexsym,amsmath,amsfonts,graphicx}
\usepackage{subfigure}
\usepackage{amssymb}
\usepackage{float}
\usepackage[dvips]{color}
\newtheorem{theorem}{Theorem}[section]
\newtheorem{thm}[theorem]{Theorem}

\newtheorem{cor}{Corollary}[section]
\newtheorem{lemma}{Lemma}[section]
\newtheorem{remark}{Remark}[section]
\newtheorem{defi}{Definition}[section]

\newtheorem{example}{Example}[section]
\numberwithin{equation}{section}

\def\cal{\mathcal }
\def\R{\mathbb R}

\def\Z{\mathbb Z}

\def\mathscr{\mathcal }

\def\diam{\text{diam}}

\newcommand{\bzero}{{\boldsymbol{0}}}

\newcommand{\bz}{{\boldsymbol{z}}}

\newcommand{\mi}{{\mathbf{i}}}
\newcommand{\bA}{{\mathbf A}}
\newcommand{\be}{{\mathbf e}}
\newcommand{\bu}{{\mathbf{u}}}

\newcommand{\ba}{{\boldsymbol{a}}}

\newcommand{\bb}{{\mathbf{b}}}
\newcommand{\bc}{{\mathbf{c}}}

\newcommand{\SD}{{\mathcal D}}

\newcommand{\SJ}{{\mathcal J}}

\begin{document}

%\title{When an integral  self-affine tile has a polygonal horn?}
\title[Tilings of convex polyhedral cones]{Tilings of convex polyhedral cones and topological properties of self-affine tiles}

%\author{Hui Rao}
%\address{Department of Mathematics and Statistics, Central China Normal University, Wuhan, 430079, China}
%\email{hrao@mail.ccnu.edu.cn}

\author{Ya-min Yang}
\address{Institute of applied mathematics, College of Science, Huazhong Agriculture of University, Wuhan,430070, China.}
\email{yangym09@mail.hzau.edu.cn}

\author{Yuan Zhang$\dag$}
\address{Department of Mathematics and Statistics, Central China Normal University, Wuhan, 430079, China\\
%School of Mathematics and  Statistic, Huazhong University of Science and Technology, Wuhan, 430074, China
}
\email{yzhang@mail.ccnu.edu.cn}

\date{\today}
\thanks{$\dag$ The correspondence author.}
%\thanks{This work is supported by NSFC nos 11601172, 11431007 and the Fundamental Research Funds for the Central Universities nos 2662015PY217.}
\thanks {This work is supported by NSFC Nos. 11431007, 11601172,
 and Fundamental Research Funds for Central Universities no.2662015PY217,
  and Self-Determined Research Funds of CCNU from the Colleges' Basic Research and Operation of MOE under Grant CCNU17XJ034.}


\begin{abstract}
{ Let $\ba_1,\dots, \ba_r$ be vectors in a half-space of $\R^n$.
We call
$$C=\ba_1\R^++\cdots+\ba_r \R^+$$ a convex polyhedral cone, and call $\{\ba_1,\dots, \ba_r\}$
a generator set of $C$. A generator set with the minimal cardinality is called a frame.
We investigate the translation tilings of convex polyhedral cones.

Let $T\subset \R^n$ be a compact set such that $T$ is the closure of its interior, and  $\SJ\subset \R^n$
be a discrete set. We say $(T,\SJ)$ is a translation tiling of $C$ if $T+\SJ=C$ and any two translations of
$T$ in $T+\SJ$ are disjoint in Lebesgue measure.

We show that if the cardinality of a frame of $C$ is larger than $\dim C$, the dimension of $C$, then $C$ does not
admit any translation tiling;
if the cardinality of a frame of $C$ equals $\dim C$, then the translation tilings of $C$ can be reduced to the translation tilings of $(\Z^+)^n$.  As an application, we characterize all the  self-affine tiles possessing
 polyhedral corners, which generalizes a result of Odlyzko [A. M. Odlyzko, \textit{Non-negative digit sets in positional number systems}, Proc. London Math. Soc., \textbf{37}(1978), 213-229.].
}
\end{abstract}

\maketitle

%\tableofcontents

% Section 1
%\input{Intro_V20} % \label{sec:intro}

\section{\textbf{Introduction}}\label{sec:intro}

Let $\ba_1,\ba_2,\cdots, \ba_m$ be $m$ non-zero vectors in a half space of $\R^n$,
that is, there is a non-zero vecter $\beta\in \R^n$ such that the inner product
$\langle \ba_j, \beta\rangle>0$ for all $j=1,\dots, m.$
 We call the set of all non-negative combinations of these vectors
$$C=\boldsymbol{a}_1\R^++\cdots +\boldsymbol{a}_m\R^+=\{\lambda_1 \boldsymbol{a}_1+\cdots+\lambda_m \boldsymbol{a}_m: \ \text{all} \ \lambda_i\geq 0\}$$
a \emph{convex polyhedral cone}.
In this case, we also say $C$ is spanned by $\{\ba_1,\dots, \ba_m\}$.

The convex polyhedral cone is an important object in convex analysis, see for instance, Rockafellar \cite{Rock}.
The main purpose of the present paper is to characterize the translation tilings of convex polyhedral cones.

  \begin{defi}\label{def:packing} {\rm Let $X\subset \R^n$, $T\subset \R^n$ be a compact set, and $\SJ\subset \R^n$ be a (finite or infinite) discrete set.

 We  say that $(T,\SJ)$ is a \emph{packing} of $X$ if
 $T+\SJ\subset X$, and $T+t_1$ and $T+t_2$ are disjoint in Lebesgue measure for any $t_1\neq t_2\in \SJ$.

  $(T,\SJ)$ is called a \emph{covering} of $X$ if $X\subset T+\SJ$.

 $(T,\SJ)$ is called a \emph{translation tiling} of $X$ if it is a packing as well as a covering of $X$.
 In this case, we call $T$ a $X$-tile and $(T,\SJ)$ a $X$-tiling.
 (In literature, usually it is assumed in addition that $T$ is the closure of the interior of $T$.)

 $(T,\SJ)$ is called a \emph{local tiling} of convex polyhedral cone $C$,
  if it is a packing of $C$ and it covers a neighborhood of $\bzero$ in $C$.
}
\end{defi}

\begin{remark}\label{rem:normal}{\rm Let $(T, {\cal J})$ be a local tiling  of   a convex polyhedral cone $C$.
Let $T+t_1$ be the tile containing $\bzero$. Set $T'=T+t_1$ and ${\cal J}'=\cal J-t_1$, then
$T'+{\cal J}'$ is a local tiling   of $C$. It follows that $\bzero\in T', \bzero\in \SJ'$ and consequently
$T'\subset C$, ${\cal J}'\subset C$.
Therefore, from now on, without loss of generality, we always  assume that
\begin{equation}\label{eq:norm}
\bzero\in T, \ \bzero\in \SJ, \  T\subset C, \text{ and } \SJ\subset C.
\end{equation}
}
\end{remark}

\subsection{Translation tilings of  convex polyhedral cones}

Let $C$ be a convex polyhedral cone.
The dimension of $C$, denoted  by $\dim C$, is the minimum of the dimensions of subspaces of $\R^n$ containing $C$.
We call $A=\{\ba_1, \dots, \ba_m\}$ a \emph{frame} of $C$, if $A$ spans $C$, and any proper subset of $A$ does not.
It is  seen that the frame $A$ of a convex polyhedral cone $C$ is unique if we require all members of $A$ to be unit vectors.

\begin{defi}{\rm We say $C$ is \emph{regular}, if
the  cardinality of a frame of $C$ equals $\dim C$, and \emph{irregular} otherwise.
}
\end{defi}

Denote $\R^+=\{x\in \R;~x\geq 0\}$
and $\Z^+=\{x\in \Z;~x\geq 0\}$.
Clearly, an $n$-dimensional convex polyhedral cone $C$ is regular if and only if $C$ is the  image of $(\R^+)^n$ under an invertible linear transformation.

We show that if $T$ can tile a `large' ball of $C$ at the origin, then not only $C$ must be regular, but also $T$ must be a union of translations of unit cubes up to a linear transformation.
Denote by $B_n(x,r)$, or simply $B(x,r)$, the ball in $\R^n$ with center $x$ and with radius $r$.

\begin{thm}\label{thm:Main-1} Let $C$ be a convex polyhedral cone.
If $(T,\SJ)$ is a local tiling of $C$ which covers $C\cap B(\bzero,R)$ for some $R>\diam(T)$,
then

$(i)$ $C$ is regular.

$(ii)$ if in addition $T=\overline{T^\circ}$, then there exist a finite set $E\subset (\Z^+)^n$,
and a linear transformation $\varphi$ of $\R^n$  such that
$
T=\varphi(E+[0,1]^n).
$
\end{thm}

Sometimes we call $(T,\SJ)$ in the above theorem a `large' local tiling.
As a consequence of Theorem \ref{thm:Main-1}, we have

\begin{cor}\label{Main-1} An irregular convex polyhedral cone   admits no translation tiling.
\end{cor}

\begin{cor}\label{Main-2}
If $(T,\SJ)$ is a tiling of $(\R^+)^n$ and $T=\overline{T^\circ}$, then there exists $E, \SJ'\subset (\Z^+)^n$,
and a positive diagonal matrix  $U=\text{diag~}(u_1,\dots,u_n)$ such that
$$
UT=E+[0,1]^n, \quad U\SJ=\SJ',
$$
and $(E, \SJ')$ is a translation tiling of $(\Z^+)^n$.
\end{cor}

Therefore, to characterize the translation tilings of regular convex polyhedral cones,
we need only characterize the translation tilings of the special cone $(\R^+)^n$,
and this can be further reduced to the problem of characterization of $(\Z^+)^n$-tilings.

%For $A, B\subset (\Z^+)^n$, we define
%$A+B=\{a+b;~a\in A, b\in B\}.$
We call $A+B$  the \emph{direct sum} of $A$ and $B$, and denoted by $A\oplus B$, if
 every element $x\in A+B$ has a unique decomposition as $x=a+b$ with $a\in A, b\in B$.
For $A, B\subset (\Z^+)^n$,  we say $(A, B)$ is a $(\Z^+)^n$-complementing pair
if $A\oplus B=(\Z^+)^n$, furthermore, we say $(A, B)$ is a $(\Z^+)^n$-tiling if $\# A<\infty$.
(We remark that it is possible that both $A, B$ are infinite sets.)

\begin{remark}{\rm
%For $A, B\subset (\Z^+)^n$, we say $(A, B)$ is a $(\Z^+)^n$-complementing-pair
%if $A\oplus B=(\Z^+)^n$. (We remark that it is possible that both $A, B$ are infinite sets.)
Rao, Yang and Zhang \cite{Yang} characterizes all the $(\Z^+)^n$-complementing pairs, which generalizes the results
 of de Bruijn \cite{deB} and Niven \cite{Niven}, which settled the case $n=1$ and $n=2$, respectively.
%  on $\Z^+$-complementing-pairs and  Niven \cite{Niven} on $(\Z^+)^2$-complementing pairs.
}\end{remark}

\subsection{Self-affine tiles possessing  polyhedral  corners}

Let $\mathbf{A}\in M_n(\mathbb{R})$ be an expanding matrix (\textit{i.e.}, all its eigenvalues have moduli larger than $1$) such that  $m=|\det(\mathbf{A})|$ is an integer larger than $1$.
  Let $\mathcal{D}=\{\mathbf{d}_0,\mathbf{d}_1,\cdots,\mathbf{d}_{m-1}\}$ be a  subset of $\mathbb{R}^n$, which we call the \emph{digit set}.
It is well known (\cite{Hut81,LW96}) that
there exists a unique non-empty compact set $T:=T(\mathbf{A},{\mathcal D})$ satisfying the set equation
\begin{equation}
T=\underset{\boldsymbol{d}\in{\cal D}}\bigcup \bA^{-1}(T+\boldsymbol{d}).
\end{equation}
We call $T(\bA,\SD)$ a \emph{self-affine tile} and $\SD$ a \emph{tile digit set},
  if $T(\bA, \SD)$  has non-void interior. A self-affine tile can   tile $\mathbb{R}^n$ by translation (\cite{LW96}).
Self-affine tiles have been studied extensively in literature (\cite{Bandt, Kenyon1992, Grochenig94,
LW96, LW96I, LW97, LauRao2001, LauLaiRao1, LauLaiRao2, Vince}),
since it is related to many fields of mathematics, such as number theory, dynamical system, spectral theory and wavelet, \textit{etc.}
 %The present paper is mainly inspired by a work of Odlyzko \cite{Od78} on feasible number systems.
As an application of Theorem \ref{Main-1} and \ref{Main-2}, we study the topological properties of $T(\bA,\SD)$.

\begin{defi}{\rm
Let $T(\bA,\SD)$ be a self-affine tile of $\R^n$.
We say $T(\bA,\SD)$ has a \emph{polyhedral corner}, if
there exists a point $x_0\in T(\bA,\SD)$, a real $r>0$,  and a convex polyhedral cone $C$,
such that
$$
B(x_0,r)\cap T(\bA,\SD)=x_0+B(\bzero,r)\cap C.
$$
}
\end{defi}

The following result generalizes a one-dimensional result of Odlyzko \cite{Od78}.
% if a self-similar tile $T(b,\SD)$ contains a polygonal horn, then it must be a finite union of $n$-dimensional cubes up to a linear transformation.

\begin{thm}\label{Main-3} If a self-affine tile $T(\bA,\SD)$ of $\R^n$
has a polyhedral corner, then there exists an affine transformation $\varphi$
such that  $\varphi(T(\bA,\SD))$ is a $(\R^+)^n$-tile.
Consequently, $T(\bA,\SD)$ is a finite union of translations of $n$-dimensional unit cubes up to an affine transformation.
\end{thm}

We close this section with some notations.
We use $\{\be_1,\dots, \be_n\}$ to denote the canonical basis of $\R^n$.
Let $\partial A$  denote the boundary of $A$,  $A^\circ$ denote  the interior of $A$, and $\overline G$ denote the closure of $G$.
 %Let $B_n(x,r)$ denote the  $n$-dimensional ball with center $x$ and radius $r$; if the dimension is obvious in %the content, we will use $B(x,r)$ instead of $B_n(x,r)$.

\medskip

The paper is organized as follows. In Sections 2--4, we show that an irregular convex polyhedral cone $C$ has
$2$-dimensional slices which are corner-cut regions.
In Sections 5-6, we show that if $C$ has a `large' local tiling, then a $2$-dimensional corner-cut slice of $C$ also has
a `large' local tiling; however, we show in Section 7 that this is impossible.  Theorem
\ref{Main-1}(i) is proved in Section 6. Section 8 is devoted to the translation tilings of $(\R^+)^n$;
 Theorem \ref{Main-1}(ii) and Corollary \ref{Main-2} are proved there. Section 9, the last section, studies the topological properties of self-affine tiles and Theorem \ref{Main-3} is proved there.

% Section 2
%\input{Cone_V20} %\label{sec:cone}
\section{\textbf{Preliminaries  on convex  polyhedral cones}}\label{sec:cone}

 First, we recall some notions about convex set,  see \cite{Rock, Ger1951}.
Let $F$ be a convex subset of the convex set $C$. We say $F$ is a \emph{face} of $C$, if  any closed line segment in $C$ with a relative interior in $F$ has both endpoints in $F$.
 An \emph{extreme point} of a convex set is one which is not a proper convex linear combination of any two points of the set.

Let $\ba_1,\ba_2,\cdots, \ba_m$ be non-zero  vectors in  $\R^n$ located in a half space. Recall that
$$C=\boldsymbol{a}_1\R^++\cdots +\boldsymbol{a}_m\R^+$$
is called a \emph{convex polyhedral cone}. Clearly $C$ is a closed set.

For a set $F\subset \R^n$, we use $\text{span}~(F)$ to denote the smallest subspace containing $F$.
Then the dimension of $F$, denoted  by $\dim F$, is  the dimension of
the subspace ${span}(F)$. Moreover, we call $F$ a $r$-face of $C$,
if $F$ is a face of $C$ with dimension $r$.

%The faces of convex polyhedral cones will play an important role in our discussion.
A convex polyhedral cone $C$ has exactly one extreme point, or $0$-face, the origin.
$1$-faces  of $C$ are the  half-lines $\ba_j\R^+$  with $\ba_j$ in the frame of $C$, and we  call $\ba_j\R^+$ an \emph{extreme direction}.

We list some facts about  faces of  convex polyhedral cones.

\begin{lemma}\label{lem:Ger}  Let $C$ be a  convex polyhedral cone with dimension $n$. Then

$(i)$  \textnormal{(Theorem 21 in \cite{Ger1951}.)} A convex cone in $\partial C$ is contained in an $(n-1)$-face $Q$.
Consequently, $\partial C$ is the union of the $(n-1)$-faces of $C$.

$(ii)$ \textnormal{(Theorem 22 in \cite{Ger1951})} Let  $Q$ be an $(n-1)$-face of $C$, then
  $$Q=C\cap \text{span}(Q).$$

$(iii)$ \textnormal{(Theorem 27 in \cite{Ger1951}.)}  If $G$ and $F$ are faces of $C$ and $F\subset G$, then $F$ is a face of $G$.

$(iv)$ If $G$ is a face of $C$, then any face of $G$ is also a face of $C$.
\end{lemma}

We shall use the following easy facts.

\begin{lemma}\label{lem:easy1}Let $C\subset \R^n$ be a convex polyhedral cone. Let $\ba,\bb\in \R^n$.

(i)  If a ray $R=\ba+\bb\R^+$ belongs to  $C$, then $\bb\in C$.

(ii) If $\ba\R^+$ is an extreme direction of $C$, and $\bb$ is a vector in $C$ but not in $\ba\R^+$,
then  $\ba-\bb\not\in C$.

\end{lemma}

\begin{proof} (i) The sequence  $(\ba+k\bb)_{k\geq 1}$  belongs to $C$ implies that
$\ba/k+\bb\in C$, hence the limit $\bb$ belongs to $C$.
 (ii)  follows from the fact that $\ba\R^+$ is a $1$-face.
\end{proof}

   Recall that a convex polyhedral cone $C$ is \emph{regular} if
the cardinality of a frame of $C$  equals $\dim C$, and \emph{irregular} otherwise.
%Clearly,  a convex cone is  regular if and only if it is the affine image of $(\R^+)^n$ under an affine map.

\begin{lemma}\label{lem:face} Let  $C$ be a regular convex polyhedral cone and let $A$ be its frame.
Then the convex polyhedral cone spanned by a subset of $A$ is a face of $C$.
\end{lemma}

\begin{proof} This follows from the definition of face.
\end{proof}

% Section 3
%\input{Slice_V20} %\label{sec:slice}

\section{\textbf{Slices of irregular convex polyhedral  cones}}\label{sec:slice}

In this section, we  investigate the intersection of a convex polyhedral  cone $C$ and a $2$-dimensional hyperplane $H$,
 where $H$ parallel to a $2$-face of $C$.

 %Clearly the intersection is a convex set, and we call this convex set a  corner-cut slice if it has more than %two extreme points. The proof of our first main result
%depends on the existence of such corner-cut slices.

%For a set $A\subset \R^n$, we use $A^\circ$ to denote the interior of $A$.

%\subsection{Corner-cut slice}
\begin{defi}{\rm Let $C$ be a convex polyhedral  cone and $F$ be a $2$-face of $C$. We call $F$ a \emph{feasible} $2$-face of $C$,
if there exists a point $x_0\in C^\circ$ such that the intersection
\begin{equation}\label{eq:slice}
(\text{span~}(F)+x_0)\cap C
\end{equation}
is a convex set with at least two extreme points; in this case, we call the set in \eqref{eq:slice} a \emph{corner-cut slice} of $C$.
}
\end{defi}

\begin{figure}
  \includegraphics[width=0.32\textwidth]{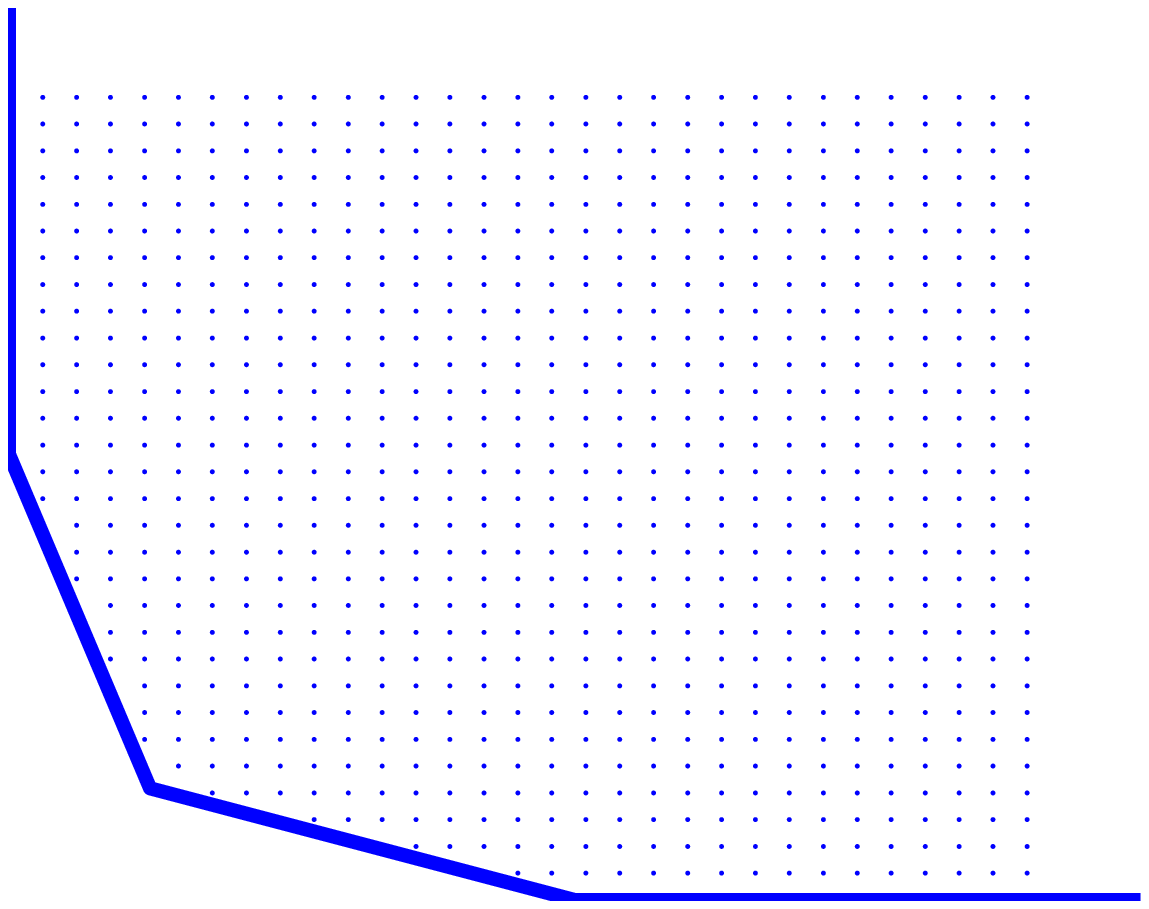}\quad
  \includegraphics[width=0.32\textwidth]{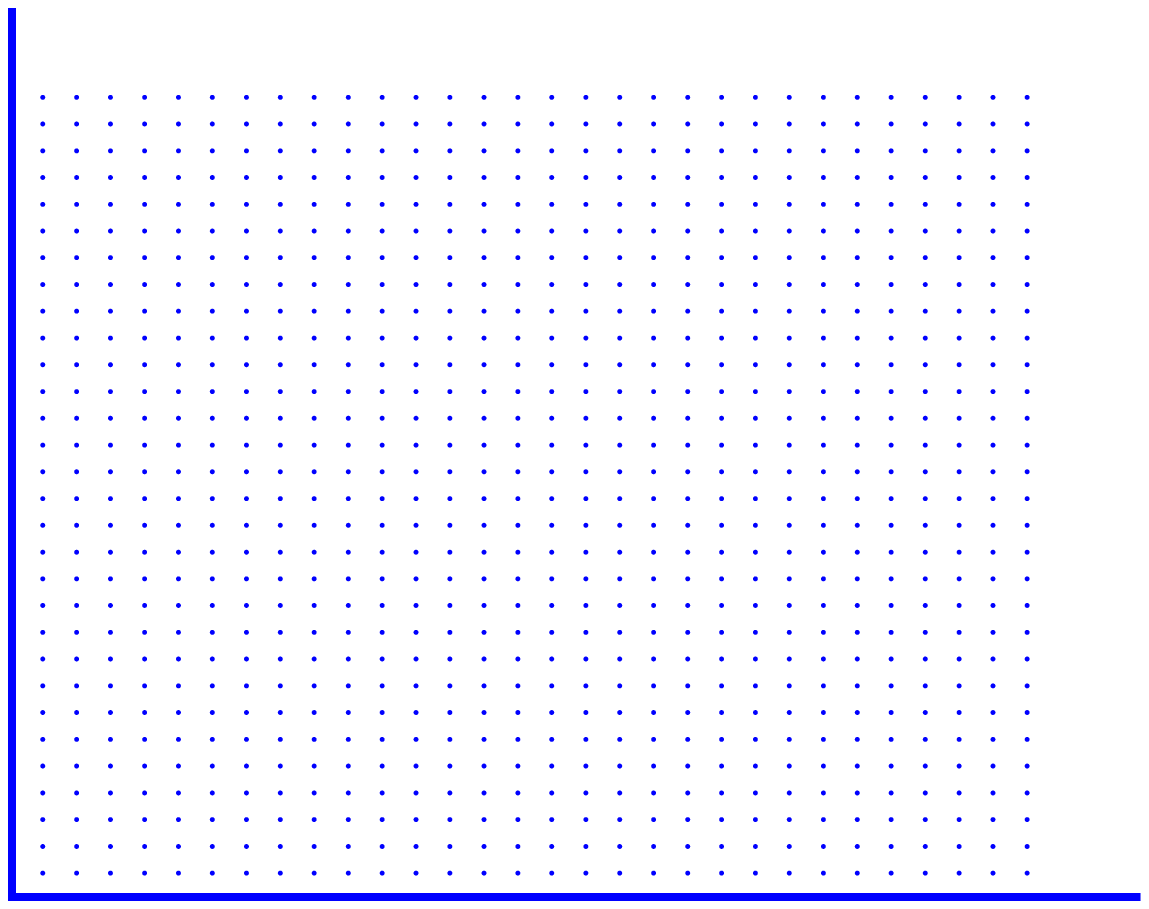}\\
  \caption{The left side   is a corner-cut slice, while the right side   is not.}\label{fig:corner}
\end{figure}

The following lemma is obvious, see Figure \ref{fig:corner}.

\begin{lemma}\label{lem:criterion}  Let $X$ be the set in \eqref{eq:slice}, then

(i) $X$ is a corner-cut slice if and only if there exists a line $L$ such that
$X\cap L$ is a line segment and $X^\circ \cap L=\emptyset$.

(ii) $X$ is a corner-cut slice if and only if there exists a ray $L$ emanating from an extreme point of $X$
such that $X\cap L$ is a line segment.
\end{lemma}

\subsection{Lemmas}
We start with several lemmas.

\begin{lemma}\label{Q-cap-H} Let $C$ be a convex polyhedral cone, and $H$
be a $2$-dimensional hyperplane intersecting the interior of $C$. Then
 $$ \partial (H\cap C)=\bigcup_Q (Q\cap H) $$
 where $Q$ runs over all the  $(n-1)$-faces of $C$.
\end{lemma}

\begin{proof} Recall that  $\partial C$ is the union of all $(n-1)$-faces, so we need only show that
\begin{equation}\label{H-C}
\partial (H\cap C)=H\cap \partial C.
\end{equation}
Let $x \in H\cap C$.
If $x\in \partial C$, then (at least) a half open ball of $B_n(x,r)$ does not belong to $C$, and hence
a half open ball of $B_2(x,r)$ does not belong to $H\cap C$, so $x\in \partial (H\cap C)$.
If $x\in C^\circ$, then clearly $x\in (H\cap C)^\circ$.
Hence \eqref{H-C} holds and the lemma follows.
\end{proof}

We use $[\ba, \bb]^+$ to denote the convex polyhedral cone spanned by $\{\ba,\bb\}$.

\begin{lemma}\label{lem:judge}  Let $C$ be a convex polyhedral cone,  $Q$ be an $(n-1)$-face of $C$,
 $[\ba,\bb]^+$ be a $2$-face of $C$, and  $y\in C^\circ$. Denote $\Sigma=\text{span}(\{\ba,\bb\})$.
 Then

  $(i)$ If $\{\ba,\bb\}\subset Q$, then $Q\cap (\Sigma+y)=\emptyset$.

  $(ii)$ If $\{\ba,\bb\}\cap Q=\{\ba\}$ or $\{\bb\}$,  then
$Q\cap (\Sigma+y)$ is either $\emptyset$, or a ray.

 $(iii)$ If $\{\ba, \bb\}\cap Q=\emptyset$, then $Q\cap (\Sigma+y)$ is $\emptyset$, or a singleton, or a line segment.

 $(iv)$ $X_y:=C\cap(\Sigma+y)$
   is a corner-cut slice if and only if there exists an $(n-1)$-face $Q$ of $C$ such that
 $Q\cap (\Sigma+y)$ is a line segment.
\end{lemma}

\begin{proof} First, by linear algebra, $Q\cap (\Sigma+y)$ is a (connected) subset of a line since $y\in C^\circ$.

(i) The first assertion holds since $(\Sigma+y)\cap \text{span}(Q)=\emptyset$.

(ii) Suppose $\ba\in Q$ and $x\in Q\cap (\Sigma+y)$. Then clearly $x+\R^+\ba$ also belongs to this intersection.
The second assertion is proved.

(iii) To prove the third assertion, we need only show that
$$
I=Q\cap (\Sigma+y)
$$
is not  a ray.
Let $L'$ be the intersection of $\Sigma$ and $\text{span}(Q)$, then $L'$ is a subspace of dimension $1$ or $2$.
 Since $\ba,\bb\not\in Q$, we have $\dim L'=1$.
  Moveover,  we have $L'=\R(\ba-c \bb)$
for some $c\in \R\setminus \{0\}$.

We claim that $c>0$, or in other words, $L'$ locates outside of the cone $C$ (except the origin).
Suppose on the contrary $c<0$, then
$$\ba-c\bb\in C\cap \text{span}(Q)=Q,$$
where the last equality is due to Lemma \ref{lem:Ger}(ii).
It follows that  $\ba, -c\bb\in Q$ since $Q$ is a face, a contradiction.
Our claim is proved.

Assume on the contrary that $I$ is a ray, and let $L$ be the line containing $I$,
then the direction of $L$ is $\pm(\ba-c\bb)$.
  By Lemma \ref{lem:easy1}(ii),  $\pm(\ba-c\bb)\not\in C$; furthermore,  by Lemma \ref{lem:easy1}(i),
   the intersection  $L\cap C$ cannot be a ray.
   It follows that the interval $I$, as a subset of $L\cap C$ is not a ray.  This  contradiction proves (iii).

 (iv) %Notice that  $Q\cap (\Sigma+y)$  must be a subset of a line since $y\not\in Q$.
Suppose $X_y$ is a corner-cut slice, then  there is a line $L$ such that
$X_y\cap L$ is a line segment and $X_y^\circ \cap L=\emptyset$ (Lemma \ref{lem:criterion}).
Clearly,
$$
(X_y\cap L)\subset \bigcup_Q Q\cap (\Sigma+y).
$$
Let $Q$ be an $(n-1)$-face such that $Q\cap (\Sigma+y)$ contains a sub-interval of $X_y\cap L$.
By Lemma \ref{lem:judge}, $Q\cap (\Sigma+y)$ is a subset of a line, hence
$Q\cap (\Sigma+y)$ is a subset of $L$, and thus a subset of
$X_y\cap L,$
 and finally must be a line segment.

On the other hand, if $I=Q\cap (\Sigma+y)$ is a line segment, let $L$ be the line containing this segment,
 then, by (iii), we have $\ba, \bb\not\in Q$, and
 the direction of $L$ is of the form $\ba-c\bb$ for some $c>0$.
%$$
%(Q\cap \Sigma) \subset (L\cap C\cap \Sigma)=(L\cap \Sigma_y).
%$$
Hence $L\cap X_y$  must be a line segment since  it cannot be a ray; moreover, since the sub-interval $I$
of $L\cap X_y$ is a subset of $\partial X_y$, $L\cap X_y$ itself must be contained in $\partial X_y$, since
$X_y$ is a planar convex set. Therefore $X_y$ is a corner-cut slice.
 The lemma is proved.
\end{proof}

\subsection{Existence of corner-cut slices}

\begin{defi}{\rm
We say a convex polyhedral cone $C$ has \emph{regular boundary}, if all its $(n-1)$-faces are regular cones. }
\end{defi}

Let $x,y\in \R^n$, we will use $\overline{[x,y]}$ to denote the line segment with endpoints $x$ and $y$.

\begin{theorem}\label{thm:slice} An irregular convex polyhedral cone with regular boundary always
    has a feasible $2$-face.
\end{theorem}

\begin{proof} We divide the proof into two cases.

\textit{Case 1.
There exists  $\bb_1,\bb_2$ in the frame of $C$ such that $[\boldsymbol{b}_1, \bb_2]^+$ is not a $2$-face of $C$.}

 Denote $G=[\bb_1, \bb_2]^+$. We claim that $G\cap C^\circ \neq \emptyset$.
Suppose on the contrary  $G\cap C^\circ = \emptyset$, then $G$ is a subset of $\partial C$, so
by Lemma \ref{lem:Ger}(i), there
exists an $(n-1)$-face $Q$ of $C$ containing $G$ .
Since $Q$ is a regular cone,
$G$ must be a face of $Q$. By Lemma \ref{lem:Ger} (iv),  $G$ is also a face of $C$, which is a contradiction.
Our claim is proved.

 By Lemma \ref{lem:Ger}(i), there exists  a $(n-1)$-face $Q$ of $C$ containing $\bb_1$.
 Using this argument repeatedly, there exists a
  $2$-face $F$ of $Q$ containing $\bb_1$.
Let $\ba_1$ be the other  element in the frame of $F$. Clearly $\ba_1$ is not a multiple of $\bb_2$,
since $F=[\bb_1, \ba_1]^+$ is a $2$-face and $[\bb_1,\bb_2]^+$ is not.

Pick any  $x_0\in [\bb_1, \bb_2]^+\cap C^\circ$.  Then $x_0$ can be written as
  $x_0=\lambda \bb_1+\rho \bb_2$,  where $\lambda,\rho>0$.
Denote $\Sigma=\text{span}(\{\ba_1, \bb_1\}).$
 We will show that
\begin{equation}X=(\Sigma+x_0)\cap C\end{equation}
is a corner-cut slice.

Clearly $\rho\bb_2$ is a extreme point
 of $X$, since it is an extreme direction of $C$. % We are looking for a second extreme point of $X$.

 Choose $\delta>0$ small so that
 $z=x_0-\delta \ba_1$ is still an interior point of $C$.
 Clearly, both $\rho \bb_2$ and $z$ belong to $X$, so $\overline {[\rho\bb_2, z]}$ also
 belongs to $X$. Let
 $$\bc=z-\rho \bb_2= \lambda\bb_1-\delta \ba_1.$$
By Lemma \ref{lem:easy1}, $\bc\not\in C$ and  the intersection  $(\rho\bb_2+\bc\R^+)\cap C$ is
not ray.
  Therefore,  the intersection of the ray $\rho\bb_2+\bc\R^+$  and  $X$ is a line segment, and by Lemma \ref{lem:criterion} (ii), $X$
 is a corner-cut slice. The theorem is proved in this case.

\textit{Case 2. For every pair  $\bb_1,\bb_2$ in the frame of $C$,  $[\boldsymbol{b}_1, \bb_2]^+$ is  a $2$-face of $C$.}

% Let $C$ be an irregular convex polyhedral cone with regular boundary.
Pick  any $(n-1)$-face $Q$ of $C$, then the frame of $Q$ has cardinality $n-1$.
Since $C$ is irregular,  we can find two elements $\bb_1$ and $\bb_2$ in the frame of $C$, but
not in the frame of  $Q$.
Let $F=[\bb_1, \bb_2]^+$,  then  $F$ is  a $2$-face of $C$ by our assumption.
Denote $\Sigma=\text{span}(F)$ and let $L$ be the intersection of $\Sigma$ and $\text{span}(Q)$, then
  $L$ is a one-dimensional subspace   since $\bb_1,\bb_2\not\in Q$,

 Pick  $x_0\in Q^\circ$. Then $y=x_0+c\bb_1\in C^\circ$ for small $c>0$.
Notice that
 $$(\Sigma+y)\cap Q=(\Sigma+x_0)\cap Q\supset(\Sigma+x_0)\cap (Q+x_0).$$
Since $L+x_0\subset \Sigma+x_0$  and $Q+x_0$   contain a small neighborhood of $L+x_0$ near $x_0$,
we deduce that $(\Sigma+y)\cap Q$   contains a line segment;
 moreover, by Lemma \ref{lem:judge}(iii), this  intersection  is a line segment since $\bb_1,\bb_2\not\in Q$.
 Finally, by Lemma \ref{lem:judge} (iv),
 $(\Sigma+y)\cap C$ is a corner-cut slice. The theorem is proved.
\end{proof}

The following example shows that Case 2 in the above proof does appear.

\begin{example}{\rm Let $\be_1,\dots, \be_6$ be the canonical basis of $\R^6$.
Let ${\boldsymbol \xi}=(1,1,1,-1,-1,-1)$. Let $C$ be the convex polyhedral cone with the frame
$\{\be_1,\dots,\be_6,{\boldsymbol \xi}\}$. Then any cone spanned by two vectors in this frame is a $2$-face.
%For simplicity, let us denote $\xi$ by $e_7$ temporarily.
%One can show that the $5$-faces are  cones spanned
%by $5$ elements in the frame  where the other two vectors $e_i$ and $e_j$ having
%index $(i,j)$ in the following set:
%$$
%(1,4),(1,5),(1,6),(1,7), (2,4),(2,5),(2,6),(2,7),(3,4),(3,5),(3,6),(3,7).
%$$
}
\end{example}

% Section 4
%\input{Metric_V9} %\label{sec:metric}

\section{\textbf{Continuity of corner-cut slices}}\label{sec:metric}
Let $C$ be an irregular  convex polyhedral cone with regular boundary.
Now we fix   a feasible $2$-face of $C$ and denote it by $F$(The existence of such a face is guaranteed by Theorem \ref{thm:slice}).
We will use the notation $X_y$ for
the slice $C\cap (\text{span}(F)+y)$ for simplicity.
The following lemma is a strengthen of Theorem \ref{thm:slice}.

\begin{lemma}\label{cor:inner} Let $C$ be an irregular convex polyhedral cone with regular boundary, and  $F$ be a feasible $2$-face of $C$.
 Then there exists a  ball   $B(x_0,r)\subset C^\circ$ such that
for any $y$ in the ball, $X_y$ is a corner-cut slice.
\end{lemma}

\begin{proof} Let $x_1$ be a point in $C^\circ$ such that $X_{x_1}$ is a corner-cut slice.
Let $u$ be the intersection of the two rays in the boundary of $X_{x_1}$;
clearly, $u\not\in C$. Choose $r_1>0$ small so that $B(u,r_1)\cap C=\emptyset$ and $B(x_1, r_1)\subset C^\circ$.

We claim that for any  $y$ belongs to  $B(x_1,r_1)\cap (C+x_1)$, a section of $B(x_1,r_1)$,
 $X_y$ is a corner-cut slice.
Since $y-x_1\in C$, we have
\begin{equation}\label{eq:move}
X_{x_1}+(y-x_1)=(\text{span~}(F)+y)\cap (C+(y-x_1))\subset X_y.
\end{equation}
Suppose on the contrary that $X_y$ has only one extreme point, then
the relation \eqref{eq:move} implies that $u+(y-x_1)\in X_y\subset C$, which contradicts $B(u,r_1)\cap C=\emptyset$. So our claim is proved.
Therefore, any  ball $B(x_0, r)\subset B(x_1,r_1)\cap (C+x_1)$ meets the requirement of the lemma.
\end{proof}

By applying a linear transformation, we may assume that $\ba$ and $\bb$, the generators of $F$, are orthogonal.

  For two sets $A, B\subset \R^n$, let $d_H(A,B)$ be the Hausdorff metric between $A$ and $B$.

  We define a mapping
  $$\pi_\ba:~ C^\circ\to \partial C$$
   as follows: For $x\in C^\circ$, by Lemma \ref{lem:easy1} (i),
  the ray $x-\R^+\ba$ intersects $\partial C$ at a single point, and we denote this point by $\pi_\ba(x)$.
Similarly, we define $\pi_\bb: ~ C^\circ\to \partial C$.

We note that, if $A$ is a subset of $C^\circ$,
$Q$ is an $(n-1)$-face of $C$, and
$\pi_\ba(A)\subset Q$, then $\pi_\ba|_A$ is the projection to $Q$ along the direction $-\ba$.
Therefore, $\pi_\ba$ is the pasting of several projections.

 \begin{thm}\label{thm:metric-one} Let $C$ be an irregular convex polyhedral cone with regular boundary, and let $F$ be a  feasible $2$-face of $C$.
 Then there exists a   ball $B^*\subset C^\circ$, such that
 for any $y\in B^*$, $X_y$ is a corner-cut slice, and
 $$  d_H(X_{y}, F),\     |a_1(y)|,\
  |b_1(y)| $$
are uniformly bounded, where $a_1(y)$ and $b_1(y)$ are the  origins of  the rays on $\partial X_y$ with direction $\ba$ and $\bb$ respectively.
 \end{thm}

 \begin{proof} Denote $\Sigma=\text{span}(\{\ba,\bb\})$, where $\{\ba,\bb\}$ is the frame of $F$.
 Let $B(x_0,r)$ be the ball such that
 $\Sigma_y\cap C$ are corner-cut slices for all $y\in B(x_0,r)$ (see  Lemma \ref{cor:inner}).
 From now on, we call a ball with this property a \emph{nice ball}.

 Notice that if $y'=y+c_1\ba+c_2\bb$  with$c_1,c_2\in \R$, then $X_y=X_{y'}$.  Consequently,
If $B$ is a nice ball, then $B+c_1\ba+c_2\bb$ is also a nice ball. Also, a sub-ball of a nice ball
is also nice.

 Let $Q_1$ be the $(n-1)$-face of $C$ such that $Q_1\cap (\Sigma+x_0)$ is the ray in $\partial X_{x_0}$
 with direction $\ba$.
  % then $\ba\in Q_1$.
  Clearly $\pi_\bb$ maps  $X_{x_0}$ to $\partial X_{x_0}$; indeed,    $\pi_\bb(x)$ is the canonical  projection of $x$ if $\pi_\bb(x)$  belongs to the ray in $\partial X_{x_0}$
 with direction $\ba$, or equivalently, belongs to $Q_1$.

  We choose  $c>0$ large so that   $\pi_\bb(x_0+c\ba)\in Q_1$, and
denote $x_1=x_0+c\ba$.
% Notice that   $B(x_0,r)+c\ba=B(x_1,r)$ is a nice ball.

Let $Q_1^\circ$ be the set of relative interior points of $Q_1$. We choose $x_2\in B(x_1,r)$
so that $\pi_\bb(x_2)\in Q_1^\circ$. (Indeed,
every point in the intersection $ Q_1^\circ+(x_1-\pi_\bb(x_1)) \cap B(x_1,r)$ fulfills
this requirement.)
   Let $r_2>0$ be a real number
such that
$$\pi_\bb(B(x_2,r_2))\subset Q_1 \text{ and } B(x_2,r_2)\subset B(x_1,r).$$

Let $Q_2$ be the $(n-1)$-face of $C$ such that
$Q_2\cap (\Sigma+x_2)$ is the ray on $\partial X_{x_2}$
 with direction $\bb$.
Similarly as above, we choose $c'$ large so that
  $\pi_\ba(x_2+c'\bb)\in Q_2$, and denote $x_3=x_2+c'\bb$.
  By the same argument as bove, there exists a ball $B(x_4, r_4)$ such that
  $$
 \pi_\ba(\overline{B(x_4,r_4)}) \subset Q_2
 \text{ and }
 \overline{B(x_4, r_4)}\subset B(x_3,r_2).
 $$
Notice that
\begin{equation}\label{eq:Q1}
 \pi_\bb(\overline{B(x_4,r_4)})\subset \pi_\bb(B(x_3, r_2))=\pi_\bb(B(x_3-c'\bb,r_2))=\pi_\bb(B(x_2,r'))\subset Q_1.
\end{equation}
%so $\pi_\bb(\overline{B(y^*,r^*)})\subset Q_1$.

Set $B^*=B(x_4,r_4)$. Clearly $\overline{B^*}$ is a nice ball. For every $y\in \overline{B^*}$,   $X_y$ is a corner-cut slice.
Moreover, by \eqref{eq:Q1},
$\pi_\bb(y)$ belongs to $Q_1\cap X_y$, which is the ray of $\partial X_y$ with direction $\ba$ ;
similarly, $\pi_\ba(y)\in Q_2$ and  locates on the ray with direction $\bb$ on $\partial X_y$.

\begin{figure}
  \includegraphics[width=0.4\textwidth]{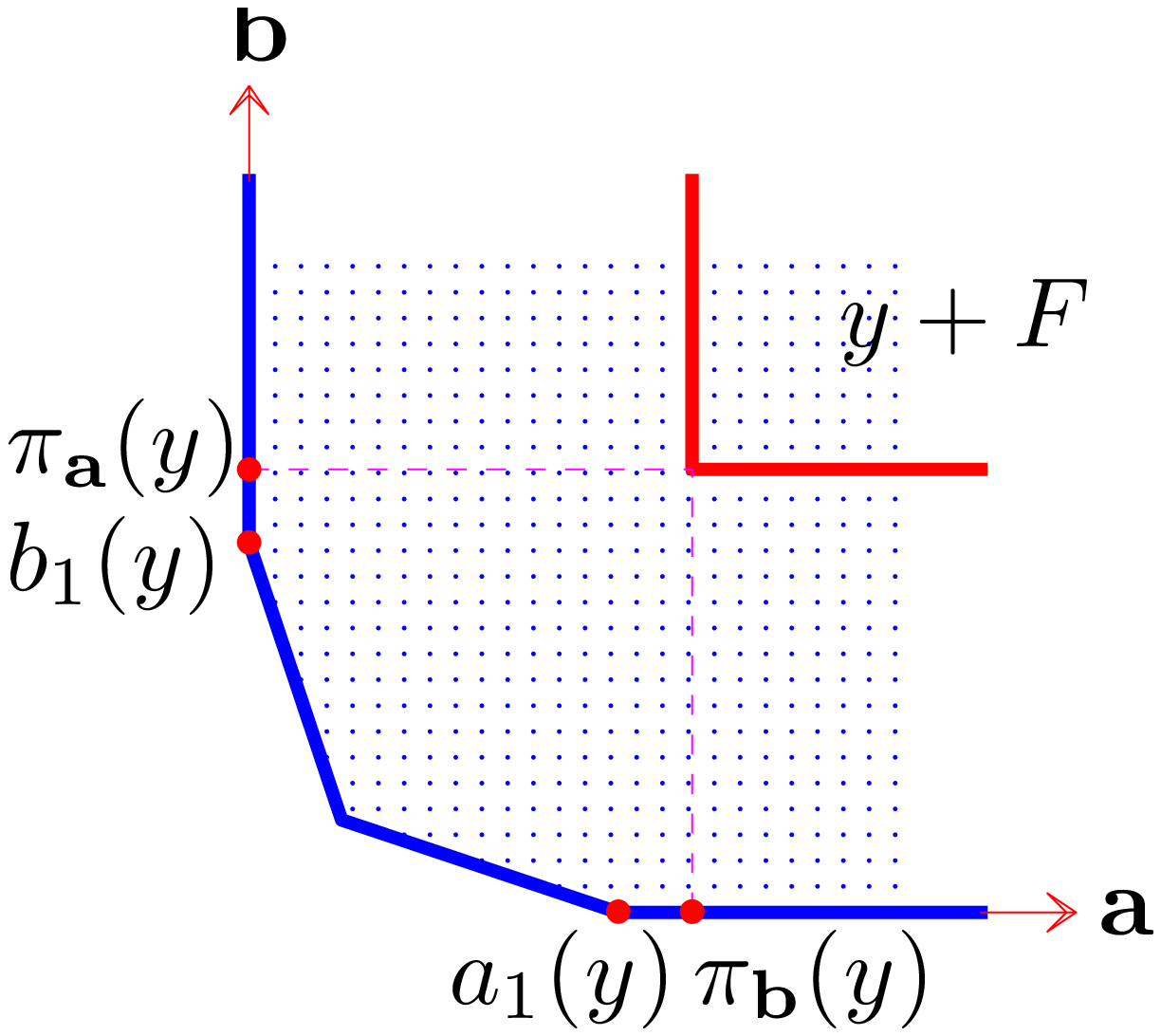}\\
 \caption{}
 \label{fig:metric}
\end{figure}

Since
$d_H(F, F+y)\leq  |y|$ and
$d_H(X_y, F)\leq d_H(X_y, y+F)+ d_H(y+F, F)$, to  show $d_H(F, X_y)$ is uniformly bounded,
we need only show that
 $$ \sup_{y\in \overline{B^*}}  d_H(X_{y}, F+y)<\infty.$$
Clearly
$$d_H(X_{y}, F+y)\leq \sqrt{|y-\pi_\ba(y)|^2+|y-\pi_\bb(y)|^2} $$
(since we assume that $\ba$ and $\bb$ are orthogonal, see Figure \ref{fig:metric},) so it is
uniformly bounded for $y\in \overline{B^*}$, since the right hand side of the above formula is a continuous function of $y$.

Finally, it is seen that both $|a_1(y)|$ and $|b_1(y)|$ are less than
$$|y|+|y-\pi_\ba(y)|+|y-\pi_\bb(y)|.$$
 The theorem is proved.
  \end{proof}

% Section 5
%\input{Shape_V9}
\section{\textbf{Polyhedral bodies}}\label{sec:shap}

Let $C$ be a convex polyhedral cone, and let $(T, {\cal J})$ be a local tiling of $C$.
 By Remark \ref{rem:normal},  we may  assume that
$\bzero \in T, \ \bzero \in \SJ, \  T\subset C, \text{ and } \SJ\subset C.$

\begin{lemma}\label{lem:unique} If $(T,\SJ)$ is a local tiling of a convex polyhedral cone $C$, then
 the origin $\bzero$ belongs to only one tile.
\end{lemma}
\begin{proof} Suppose on the contrary that two tiles $T+a_1$ and $T+a_2$ both contain $\bzero$.
Then $-a_1,-a_2\in T$. So $a_1-a_2\in T+a_1$ and $a_2-a_1\in T+a_2$. Therefore, both $a_1-a_2$ and $a_2-a_1$ belong to $C$, which is a contradiction.
\end{proof}

We call a set $\Omega$ a \emph{polyhedral  corner}, if there exist a point $x$,
a number  $r>0$ and a convex polyhedral cone $D$ such that
 $$\Omega=x+ B(\bzero,r)\cap D,$$
 and  we call $x$ the \emph{vertex} of the polyhedral corner.

 \begin{defi}{\rm
Let $T\subset \R^n$ be a compact  set. For a point $x\in \partial T$, if
there exists a real $r>0$, such that  $B_n(x,r)\cap T$ is a non-overlapping union
of several   polyhedral corners with the same vertex $x$,  then we call $x$  a `nice' point of $T$;
otherwise, we call $x$ a `bad' point of $T$.
If all points in $\partial T$ are `nice', then we call $T$ a \emph{polyhedral body}.
}
\end{defi}

 If $C$ is a convex polyhedral cone and $x\in \partial C$, then $B(x,r)\cap C$ is a finite non-overlapping union of  polyhedral corners
 for $r$ small.

  \begin{lemma}\label{lem:corner} Let $D_1,\dots, D_k$ be convex polyhedral cones of dimension $n$ such that their interiors are disjoint, then

 $(i)$ $B(\bzero,r)\setminus(D_1\cup\cdots \cup D_k)$ is a finite non-overlapping union of polyhedral corners.

 $(ii)$ Let $\Omega=B(\bzero,r)\cap (D_1\cup\cdots\cup D_k)$
and let $\nu$ be the $(n-1)$-dimensional Hausdorff measure, then for $\nu$-almost every
point $x\in \partial \Omega$ with $|x|<r$, there exists a real number $\delta>0$ such that
$\overline{B_n(x,\delta)}\cap \Omega$ is a closed half ball.
 \end{lemma}

 \begin{proof}  Each $D_j$ is bounded by a set of subspaces. Let $S_1,\dots, S_m$ be the collection of such spaces for $j=1,\dots, k$.

 (i) Let $S_{m+1}, \dots, S_{m+n}$ be the subspaces $\{(x_1,\dots, x_n)\in \R^n;~x_j=0\}$, $j=1,\dots, n$. Then
 $S_1,\dots, S_{m+n}$ decompose $\R^n$ into at most $2^{m+n}$ non-overlapping convex polyhedral cones, which we denote
 by $C_1,\dots, C_h$. Then
 $$
 B(\bzero,r)\setminus(D_1\cup\cdots \cup D_k)=\bigcup \left \{B(\bzero,r)\cap C_j;~ C_j \text{ is not a subset of } D_1\cup\cdots \cup D_k \right \}.
 $$

 (ii) Let $x\in \partial \Omega$. Then $x\in S_1\cup\cdots \cup S_m$.

 Suppose $x$ belongs to only one element in $\{S_1,\dots, S_m\}$, say,  $x\in S_{j'}$. Let $\delta$ be a real number such that
 $d(x, S_i)>\delta$ for all $i\neq j'$. The ball $B_n(x, \delta)$ is cut into two (closed) half balls by $S_{j'}$. For any $D_j$, $j\in \{1,\dots, k\}$, either $D_j$ contains a half  ball of $B(x,\delta)$, or disjoint with $B(x,\delta)$. Moreover, only one $D_j$ intersects $B(x, \delta)$, for otherwise, the two half balls
 of $B(x,\delta)$ belong to two different $D_j$, $j=1,\dots, k$, and so $x\in \Omega^\circ$, which is absurd.
 Therefore, $\overline{B(x,\delta)}\cap \Omega=\overline{B(x,\delta)}\cap D_{j}$ and it is a closed half ball.

 Finally, notice that $S_i\cap S_j$ is a $\nu$-zero set, the lemma is proved.
 \end{proof}

\begin{thm}\label{thm:body} Let $C$ be a convex polyhedral cone.  If $(T, {\cal J})$
is a local tiling of $C$ which covers  $C\cap B(\bzero, R)$ with $R>\diam(T)$, then  $T$ must be a polyhedral body.
\end{thm}

\begin{proof} By Lemma \ref{lem:unique},
  an open neighborhood ${\cal N}$ of $\bzero$ is contained in $T$, and ${\cal N}\cap \partial T={\cal N}\cap \partial C$, hence every
$x\in {\cal N}\cap \partial T$ is a `nice' point in $T$.
Also, notice that   $x\in \partial(T+t)$ is a `bad' point of the tile $T+t$ if and only if
$x-t\in \partial T$ is a `bad' point of the tile $T$.

 Suppose on the contrary   that the set of `bad' point in $\partial T$, which we denote by $G$, is not empty.
Let $\beta$ be a non-zero vector in $\R^n$ such that $\langle x, \beta\rangle>0$ for all $x\in C\setminus\{\bzero\}$.
Let $z_0$ be a point in $\overline G$  such that $\langle z_0, \beta \rangle$ attains the minimal value.
Let
$$
\epsilon=\min\{\langle t, \beta \rangle;~t\in \SJ\setminus\{\bzero\}\}.
$$
Clearly $\epsilon>0$ by the discreteness of $\SJ$.
Let   $z$ be a point in $G$ such that $|z-z_0|<\frac{\epsilon}{|\beta|}$;  a simple calculation shows that $\langle z-t, \beta\rangle< \langle z_0, \beta\rangle $ for any $t\in \SJ\setminus \{\bzero\}$.

If there is only one tile in $T+\SJ$ covering $z$, then for a small $r>0$,  $B_n(z,r)\cap T=B_n(z,r)\cap C$, which implies that $z$ is a `nice' point, a contradiction.

If there are exactly two tiles in $T+\SJ$ covering $z$, say, $z\in T\cap (T+t_1)$ where $t_1\neq \bzero$,
then
$$B_n(z,r)\cap C=B_n(z,r)\cap (T\cup (T+t_1))$$
for $r$ small. So, by Lemma \ref{lem:corner},
   $z$ is a `bad' point of
$T+t_1$ since $z$ is a `bad' point of $T$,  and $z-t_1$ is a `bad' point of $T$.
Hence $z-t_1\in G$ and $\langle z-t_1,\beta\rangle<\langle z_0, \beta\rangle$,
which contradicts the minimality of $z_0$.

If there are exactly $k+1$ number of tiles  covering $z$,
say, $z\in T\cap (T+t_1)\cap \dots \cap (T+t_k)$, then   by Lemma \ref{lem:corner},  for at least one $1\leq j\leq k$,
$z$ is  a `bad'  of $T+t_j$. By the same argument as above, we get  a contradiction.
The theorem is proved.
\end{proof}

Next, we show a local tiling of $C$ induces a local tiling of $Q$ for every $(n-1)$-face $Q$.

\begin{thm}\label{boundary tiling}   Let $C\subset \R^n$ be a convex polyhedral cone, and
$(T,\SJ)$ be a local tiling of $C$  which covers  $C\cap B(\bzero, R)$ with $R>\diam(T)$. Then    for any $(n-1)$-face $Q$ of $C$,
 $(T\cap Q, \SJ\cap Q)$ is a local tiling of $Q$ which covers $Q\cap B(\bzero, R)$.
\end{thm}

\begin{proof} By Theorem \ref{thm:body}, $T$ is  a polyhedral body. Let $\nu$ be the $(n-1)$-dimensional Hausdorff measure. We claim that  for $\nu$-almost every
point $x\in \partial T$, there exists a real number $\delta>0$ such that
$\overline{B_n(x,\delta)}\cap T$ is a closed half ball.
By compactness of $T$, there exists $x_j\in \partial T$, $r_j>0$, $j=1,\dots, k$, such that
each $B(x_j,r_j)\cap T$ is a finite non-overlapping union of  polyhedral corners and $\partial T$ is covered by
$\{B(x_j, r_j)\cap T; ~j=1,\dots, k\}$.
So our claim holds by Lemma \ref{lem:corner}(ii).

Since
$$
(T+\SJ)\cap Q=(T\cap Q)+ (\SJ\cap Q),
$$
we have that  $(T\cap Q, \SJ\cap Q)$ is   a covering of $Q\cap B(\bzero,R)$.
%Denote $\SJ_0=\{t;~t\in Q\}$.
To prove the theorem, it suffices to show that $(T\cap Q, \SJ\cap Q)$ is a packing of $Q$, that is, the
intersection
\begin{equation}\label{eq:T-cap-Q}
((T\cap Q)+t_1) \cap ((T\cap Q)+t_2)
\end{equation}
is a $\nu$-zero set   for $t_1, t_2\in \SJ\cap Q$ and $t_1\neq t_2$.

 Suppose the intersection \eqref{eq:T-cap-Q} is not a $\nu$-zero set, then by the claim above,  there exists  a point   $x$ in the above intersection
 such that both $\overline{B_n(x,\delta)}\cap (T+t_1)$ and $\overline{B_n(x,\delta)}\cap (T+t_2)$
 are closed half balls for $\delta$ small enough.
  It follows that the closed half ball $\overline{B_n(x,\delta)}\cap C$ coincide with the above two half balls, and
  is a subset of both $T+t_1$ and $T+t_2$, a contradiction. The theorem is proved.
\end{proof}

% Section 6
%\input{Irregular_V20} %\label{sec:irregular}
\section{\textbf{Proof  of Theorem \ref{Main-1}(i)}}\label{sec:irregualr}

To prove Theorem \ref{Main-1}(i), we need only show that
every irregular convex polyhedral cone with regular boundary admits no translation tiling.
For if $C$ is an irregular cone whose boundary is not regular and if $C$ admits a `large' local tiling $(T, \SJ)$, then
there is an $(n-1)$-face $Q$ of $C$ is irregular, and  by Theorem \ref{boundary tiling}, $Q$ also admits a
`large' local tiling. Therefore, Theorem \ref{Main-1}(i) can be proved by induction.

\indent  {\textbf{Assumptions}} \textit{ In the rest of this section, we always assume that  $C$ is an irregular convex polyhedral cone
with regular boundary,
and $(T,\SJ)$ is a packing  of $C$ as well as  a covering of $C\cap B(\bzero, R)$ with $R>\diam(T)$.}

  Under the above assumptions,  we are going  to deduce a contradiction.

\indent {\textbf{Notations}} \emph{ Let $F$ be a feasible $2$-face of $C$ with frame $\ba$ and $\bb$. We may assume that $\ba$ and $\bb$ are orthogonal by applying a linear transformation.
 Let $B^*\subset C^\circ$ be a ball in Theorem \ref{thm:metric-one}, that is,
for any $y\in B^*$,
$$X_y=(\text{span}(F)+y)\cap C$$
 is a   corner-cut slice, and
 \begin{equation}\label{eq:N}
  N=\sup_{y\in  B^*} d_H(X_y, F)<\infty,
  \end{equation}
  \begin{equation}\label{eq:M}
  M=\sup_{y\in  B^*} |a_1(y)|<\infty,\ \   M'=\sup_{y\in  B^*} |b_1(y)|<\infty,
  \end{equation}
  where $a_1(y)$ and $b_1(y)$ are the origins of the rays on $\partial X_y$ with directions $\ba$ and $\bb$, respectively. Denote   $\Sigma=\text{span}(F)$.
 }

\begin{lemma}\label{lem:dilation} Let $\delta>0$.
For any  $y\in \delta B^*$, $X_y$ is a corner-cut slice and
  \begin{equation}\label{eq:dilation}
  \sup_{y\in \delta B^*} d_H(X_y, F)=\delta N,
  \end{equation}
  where $N$ is defined in \eqref{eq:N}.
\end{lemma}

\begin{proof} Since
$$
X_y=(\Sigma+y)\cap C=\delta\left ((\Sigma+y/\delta)\cap C\right )=\delta X_{y/\delta},
$$
that is, $X_y$ is a dilation of $X_{y/\delta}$, we infer that  $X_y$ is  a corner-cut slice for $y\in \delta B^*$. \eqref{eq:dilation} follows from the fact that $d_H(\delta A, \delta B)=\delta d_H(A,B)$.
\end{proof}

 Now we study the intersection of $\Sigma+y$ and the local tiling $(T, \SJ)$.
We define the \emph{$x$-section} of a set $A$ as
\begin{equation}\label{x-slice}
(A)_{x}=A\cap (\Sigma+x).
\end{equation}

Here are some easy facts.

\begin{lemma}\label{lem:up}
(i)  If $t\in \SJ\setminus F$, then $(T+t)\cap F=\emptyset$.

(ii) For $t\in \SJ\cap F$, we have $(T+t)_{x}=(T)_{x}+t$.
\end{lemma}

\begin{proof} (i) For otherwise, there exists $x\in T$ and $t\in \SJ\setminus F$ such that
$x+t\in F$. So $x/2+t/2$, a convex combination of $x$ and $t$ belongs to $F$.
 By the definition of a face, $x, t\in F$, which is a contradiction.

 (ii)  Since  $(T+t)_{x}=(T+t)\cap (\Sigma+x)=(T\cap (\Sigma-t+x))+t=T\cap (\Sigma+x)+t=(T)_{x}+t.$
\end{proof}

Let $\mu_r$ be the $r$-dimensional Lebesgue measure.

\begin{lemma}\label{lem:packing}  For any $\delta>0$, there exists $y\in \delta B^*$ such that
$((T)_y, \SJ\cap F)$ is a packing of $X_y$.
% that is, $(T+t_1)_{y}$ and $(T+t_2)_{y}$
% are disjoint in measure $\mu_2$  provided $t_1\neq t_2\in \SJ\cap F$.
 \end{lemma}

 \begin{proof}  Fix $t_1,t_2\in \SJ\cap F$ with $t_1\neq t_2$.
 Notice that $T+t_1$ and $T+t_2$ are disjoint in  measure $\mu_n$.
Let  $M\subset \delta B^*$ be  a cube of dimension $n-2$
 such that $M$ is  orthogonal to $\Sigma$.
Let $f: C^\circ \to \R$ be the function
$$
f(x)=\mu_2((T+t_1)_x\cap (T+t_2)_x).
$$
By Fubini Theorem,
$$
\int_{x\in M} f(x)dx\leq \mu_n((T+t_1)\cap (T+t_2))=0,
$$
so $f(x)=0$ \textit{a.e.} $x\in M$.

Hence, for a pair $t_1,t_2\in\SJ$, to insure $\mu_2((T+t_1)_x\cap (T+t_2)_x)=0$, we need to eliminate
a measure zero set of $M$. After eliminating the measure zero sets for all pairs  in $\SJ\cap F$, a point $y$ in the remaining set   fulfills the requirement of the lemma.
 \end{proof}

%Denote $\SJ_F=\SJ\cap F$. We have just proved that $T+\SJ_F$ covers $F$.

\begin{lemma}\label{lem:covering} There exists  $\delta>0$  such that for any   $y\in  \delta B^*$,
$B(\bzero, R)\cap X_{ y}$ is covered by $(T)_{ y}+(\SJ\cap F)$.
\end{lemma}

 \begin{proof}
Without loss of generality, we may assume that $\SJ$ is  a finite set.
Let
$$
\varepsilon=\min\{d(T+t, F);~ t\in \SJ\setminus F \}.
$$
By Lemma \ref{lem:up},   $d(T+t, F)>0$ for all $t\in \SJ\setminus F$, so $\varepsilon >0$.
In other words, we have that
\begin{equation}\label{eq:far}
d(T+t, F)\geq \varepsilon \quad  \text{ for all } t\in \SJ\setminus F.
\end{equation}

We choose $0<\delta<\varepsilon/N$,  where $N$ is defined in \eqref{eq:N}. Pick any $y\in \delta B^*$, by Lemma
\ref{lem:dilation},
\begin{equation}\label{eq:dF}
d_H(X_{  y}, F)<\varepsilon .
\end{equation}
Equation \eqref{eq:far} together with \eqref{eq:dF} imply that
$$(T+t)\cap X_{ y}=\emptyset \text{ for } t\in \SJ\setminus F,$$
so
$ X_{y} \cap B(\bzero, R )$ is covered by  $(T+t)_y$ with
 $t\in \SJ \cap F$. Finally,   $(T+t)_y=(T)_y+t$ by Lemma \ref{lem:up}(ii). The lemma is proved.
\end{proof}

\medskip
\noindent\textbf{Proof of Theorem \ref{thm:Main-1}(i).}
% Let $C$ be an irregular convex polyhedral cone.
We prove assertion (i) of the theorem by induction on the dimension of $C$.

For $\dim C=1$ or $2$, the cone must be regular, and the theorem holds automatically.

Suppose $\dim C=n$. Assume on the contrary that $C$ is irregular, $(T,\SJ)$ is a packing of $C$ as well as a covering of $C\cap B(\bzero,R)$ with $R>\diam(T)$.

If $C$ has irregular boundary, then an $(n-1)$-face $Q$ of $C$ is an irregular cone.
By Theorem \ref{boundary tiling}, $(T\cap Q, \SJ\cap Q)$ is a
packing of $Q$ and a covering of $Q\cap B(\bzero, R)$, which is   impossible by our induction hypothesis.
So $C$ must be an irregular cone with regular boundary. Now we use the notations listed in the beginning of this section.

Let $0<\kappa<R-\diam(T)$.
Let $\delta>0$ be the constant in Lemma \ref{lem:covering} and satisfies the additional requirement
$$
\delta<\min \left \{ \frac{\kappa}{M}, \frac{R-\kappa}{M+M'} \right \},
$$
where $M$ and $M'$ are the constants in formula \eqref{eq:M}.

Let $y$ be a point in $\delta B^*$ satisfying the requirements of Lemma \ref{lem:packing}.
Then members in the cluster
\begin{equation}\label{eq:cluster}
\{(T)_y+t;\ t\in \SJ\cap F\}
\end{equation}
are disjoint in $\mu_2$, and cover the set $B(\bzero, R)\cap X_y$.

 Since $\delta<\kappa/M$, by \eqref{eq:M} and $X_y=\delta X_{y/\delta}$,  for any $y\in \delta B^*$,
$|a_1(y)|<\kappa$ where $a_1(y)$ is the origin
 of the ray on  $\partial X_y$ with direction $\ba$.
  Then $B(a_1(y), R-\kappa)\subset B(\bzero, R)$ and hence the cluster in \eqref{eq:cluster} is a packing of $X_y$ and covers the set $B(a_1(y), R-\kappa)\cap X_y$.
  Similarly, from $\delta<(R-\kappa)/(M+M'))$, we deduce that
  $R-\kappa>|a_1(y)-b_1(y)|$ for any $y\in \delta B^*$.
However, in next section, we prove that this is impossible (Theorem \ref{thm:tiling}), and we get a contradiction. This completes the proof of the theorem.
$\Box$
\medskip

%From the proof of the above theorem and using \ref{cor:local-1}, we also have that
%
%\begin{cor}\label{cor:local-2} Let $C\subset \R^n$ be a convex polyhedral cone.
%Let $T$ be a compact set, $\SJ$ be a discrete set and  $R>2 \diam(T)$.
%If $(T, \SJ)$ is a packing of $C$ and it covers  $C\cap B(0, R)$, then for any $(n-1)$-face $Q$ of $C$,
%then  $C$ is regular.
%\end{cor}

% Section 7
%\input{Cut_V20} %\label{sec:cut}
\section{\textbf{Corner-Cut region can not be tiled by translations of  one set}}\label{sec:cut}

%\subsection{\textbf{Corner-Cut region}}
Let $X\subset \mathbb{R}^2$ be a unbounded convex region determined by a system of liner inequalities:
 \begin{equation}
 a_jx+b_j y+c_j\geq 0, \quad 1\leq j\leq N.
 \end{equation}
We call $X$ a \emph{corner-cut region}, if  $X$ has at least two extreme points, and the two rays
on $\partial X$  are not parallel.

Let $a_1, a_2,\cdots, a_q$ be the extreme points of $\partial X$ from left to right, and write
$$\partial X= \ell_0 \cup \ell_1 \cup \cdots \cup \ell_q,$$
where $\ell_0$ and $\ell_q$ are two rays, and the other $\ell_j=\overline{[a_j, a_{j+1}]}$ are line segments.

\begin{theorem}\label{thm:tiling}
Let $X\subset \mathbb{R}^2$ be a corner-cut region,  $T$ be a compact set, and  $\SJ\subset \R^2$ be a finite set.
Let
$$R>\max\{\text{diam}(T), |a_1-a_2|\}.$$
Then   the following two items can not be fulfilled at the same time:

(i)  $X\cap B(a_1,R)$ is covered  by $T+\SJ$ and $T+\SJ\subset X$;

(ii) The members in $\{T+t;~t\in \SJ\}$ are disjoint in Lebesgue measure.
\end{theorem}

Note that in the above theorem, we do not assume that $T=\overline{T^\circ}$.

 \begin{proof}
 By applying an affine map,
 without loss of generality, we may assume that $a_1=\bzero$ is the origin, and the two rays
  on $\partial X$ are $\ell_0=\{0\}\times [0,+\infty)$ and $\ell_q=a_q+ [0,+\infty)\times \{0\}$.
  (We remark that under this assumption, it holds that
  $
  |a_1-a_2|\leq |a_1-a_q|.
  $
So we can use Theorem \ref{thm:tiling} in the previous section.)
For simplicity, we identify $\R^2$ to the complex plane $\mathbb C$.

Suppose on the contrary that  $(T,\mathcal{J})$ is a pair satisfying the two items in Theorem \ref{thm:tiling}.
By Remark \ref{rem:normal}, without loss of generality, we may assume $\bzero\in T$,  $\bzero\in \mathcal{J}$,
   $\mathcal{J}\subset X$, and $T\subset X$.

\begin{lemma}\label{extreme}
For any $b\in \{a_1,\cdots,a_q\}$, $b$ belongs to exact one tile in $\{T+t:t\in \mathcal{J}\}$.
Especially, $T$ is the only tile contains the point $a_1=\bzero$.
\end{lemma}

\begin{proof}
Assume that $b\in (T+t_1)\cap (T+t_2)$, where $t_1,t_2\in \mathcal{J}$ and $t_1\neq t_2$. Then
$b-t_1,b-t_2\in T$. So
$$z_1=b-t_1+t_2\in T+t_2\subset X \text{ and }z_2=b-t_2+t_1\in T+t_1\subset X.$$ Hence $b=({z_1+z_2})/{2}$,
which contradicts that $b$ is an extreme point.
\end{proof}

%For two points $a,b\in \mathbb{R}^2$, we denote the segment from $a$ to $b$ by $\overline{[a,b]}$.
The following is a technical lemma we need in the proof of the following Lemma \ref{lem:nb}.

 \begin{lemma}\label{lem:trapezoid} Let $A\subset \R^2$ be a trapezoid, and let $L=[0,1]\times\{0\}$ be the base line of $A$ with shorter length. Let $T$ be a compact subset of $A$, and $\SJ=\{t_0=\bzero, t_1,\dots, t_p\}$ be a subset of $L$ with $p\geq 1$. Then
$(T, \SJ)$ can not be a tiling of $A$.
 \end{lemma}
The proof of the above lemma is very similar to the proof of Theorem \ref{info-2}, but much more simpler.
We put the proof in Appendix A.
In the following, the topology we use is  the relative topology of $X$.

\begin{lemma}\label{lem:nb}
A neighborhood of $\overline{[a_1,a_2]}\subset T$.
\end{lemma}

\begin{proof} By our assumption, $\overline{[a_1,a_2]}\subset B(\bzero,R)$ and is covered by $(T,\SJ)$.

Let $t_0=\bzero,t_1,\cdots,t_p$ be the elements in $\mathcal{J}$ satisfying
$$(T+t_j)\cap \overline{[a_1,a_2]}\neq \emptyset.$$
Then $t_0,t_1,\cdots,t_p\in \overline{[a_1,a_2]}$, since $\overline{[a_1,a_2]}$ is a face of $X$.

To prove the lemma, we need to show that $p=0$. Suppose on the contrary $p\geq 1$.

 Assume that $t_0,t_1,\cdots, t_p$ are arranged from left to right on $\overline{[a_1,a_2]}$ .
Let $L$ be the line containing $\overline{[a_1,a_2]}$.
There exists $\delta>0$ such that no element of $\SJ\setminus\{t_0,t_1,\dots,t_p\}$
locates in strip between $L$ and $L+\delta{\mathbf i}$. Denote
$$
A=X\cap (L+[0, \delta]\cdot {\mathbf i}),
$$
then using $(T+t_j)\cap (L+[0, \delta]\cdot {\mathbf i})=(T+t_j)\cap A$, it is easy to show that
$$
A=(T\cap A)+\{t_0,\dots, t_p\},
$$
and the right hand side is a tiling of $A$.
We choose $\delta$ small to ensure  $A$ is a trapezoid. By Lemma \ref{lem:trapezoid},
this is impossible.
The lemma is proved.
\end{proof}

Let $z_0\in \partial X$ be the first point on the right side of $a_1$ such that $z_0$ is not a relative interior point of  $T$, that is,  there exists $t^*\in \mathcal{J}\setminus \{\bzero\}$, such that $z_0\in T+t^*$.
 Then $z_0\not\in \{a_1,\cdots,a_q\}$ by Lemma \ref{extreme} and
 $z_0$ is on the right side of $a_2$ by Lemma \ref{lem:nb}.

 Let $\gamma$ be the open broken line from $a_1$ to $z_0$ on $\partial X$. Since $z_0-t^*\in T\subset X$ and $t^*\in \mathcal{J}\subset X$, their real parts must be non-negative, \textit{i.e.},
 \begin{equation}\label{eq:positive}
 Re(z_0-t^*)\geq 0 \text{ and } Re(t^*)\geq 0.
 \end{equation}

 If $Re(z_0-t^*)=0$, then $t^*$ and $z_0$ are located in the same vertical line, so to guarantee $z_0\in T+t^*$,
 we must have $z_0=t^*$. Since the slope of the line containing $z_0$ is larger than the slope of the
 line containing $\overline{[a_1,a_2]}$, we get that $T+t^*$ is not a subset of $X$, a contradiction.
 If $Re(t^*)=0$, then $T+t^*$ can not contain $z_0$, which is also a contradiction.
 So the equalities in \eqref{eq:positive} are strict, and consequently,
 \begin{equation}\label{middle-1}
 0< Re(z_0-t^*)< Re(z_0),
 \end{equation}
 \begin{equation}\label{middle-2}
 0<Re(t^*)<Re(z_0).
 \end{equation}
Moreover, since $z_0$ is a lowest point in $t^*+T$ (w.r.t. the vertical direction), $z_0-t^*$ must be a lowest point in $T$ (w.r.t. the vertical direction),
hence, by formula \eqref{middle-1}, we have
%  $z_0-t^*$ must be located on the curve $\gamma$, \emph{i.e.},
  \begin{equation}
  z_0-t^*\in \gamma^\circ.
  \end{equation}

\begin{figure}
  % Requires \usepackage{graphicx}
  \includegraphics[width=0.6\textwidth]{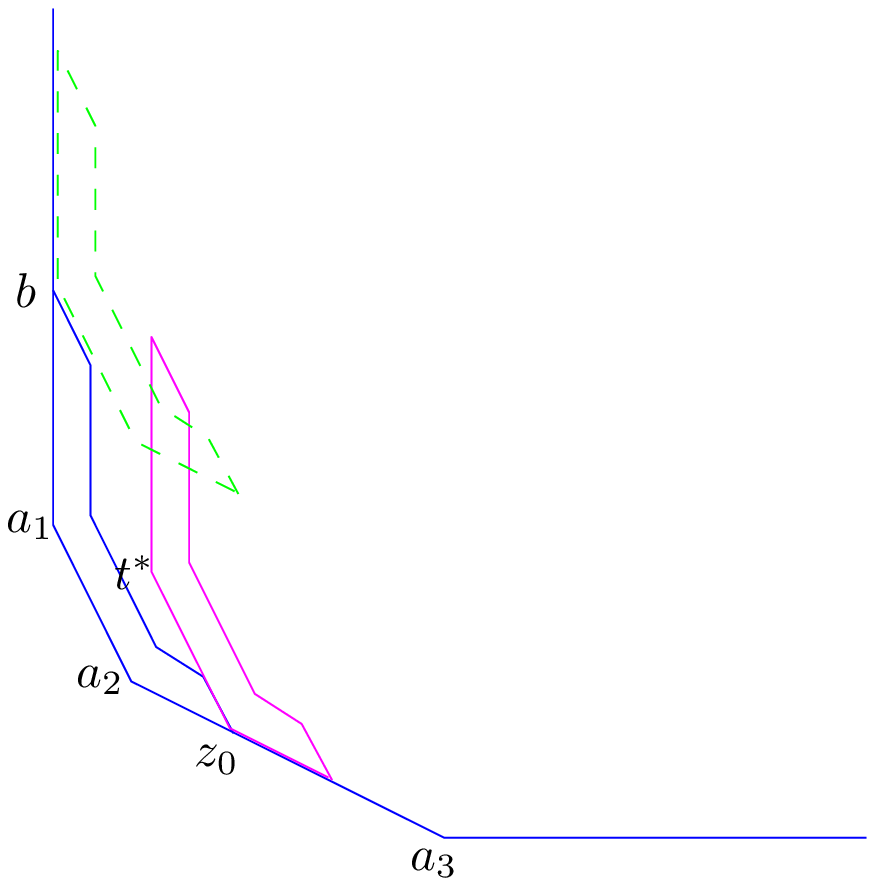}\\
  \caption{}
  \label{fig:U}
\end{figure}

Let $b\neq \bzero$ be the smallest point on $\ell_0\cap \mathcal{J}$.
Clearly the interval $\overline{[b, a_1]}\setminus \{b\}$ is contained in $T$ and does not intersect any other tiles. Denote by $\kappa$  the open
broken line from $b$ to $z_0$ on $\partial X$, then there is an open set $U$ such that
$$
\kappa\subset U\subset T.
$$
See Figure \ref{fig:U}, where we use a polygon to illustrate the open set $U$.

  Next we show that
  \begin{equation}\label{eq:cross}
  (b+\kappa)\cap (t^*+\kappa)\neq \emptyset.
  \end{equation}
   We claim that
  $t^*+b$, the initial point of $t^*+\kappa$, is above $b+\kappa$,
   and $b+z_0$, the terminal point of $b+\kappa$, is above $t^*+\kappa$.
   The first assertion holds, since by formula \eqref{middle-2}, $t^*$ is above the curve $\kappa$, and the second assertion holds
   since the point $b+(z_0-t^*)$ is above $\kappa$. Our claim is proved.
   If $t^*$ is below the curve $b+\gamma$, then apparently \eqref{eq:cross} holds.
   If $t^*$ is above $b+\gamma$, regarding $b+\gamma$ and $t^*+\gamma$ as graphs of two functions and using the Intermediate Value Theorem, we conclude
    $(b+\gamma)\cap (t^*+\gamma)\neq \emptyset$ and \eqref{eq:cross} is confirmed.

     Let $x$ be the intersection of the two curves in \eqref{eq:cross},
    then $x+[0,\delta]^2$ belongs to both $b+U$ and $t^*+U$ for $\delta$ small.
  It follows that  $(b+T)^\circ\cap (t^*+T)^\circ \neq \emptyset$, which is a contradiction.
 \end{proof}

\section{\textbf{From tilings of $(\R^+)^n$ to tilings of $(\Z^+)^n$}}\label{sec:regular}

Let $T$ be a compact set satisfying $T=\overline{T^\circ}$. Let $R>\diam(T)$.
 Let $(T,\SJ)$ be a packing of $(\R^+)^n$ as well as a covering of $(\R^+)^n\cap B(\bzero,R)$; we call
such  $(T,\SJ)$ a \emph{local tiling} in this section.
 As before, we  may assume that
 $\bzero\in T,$ $\bzero\in \SJ,$  $T\subset (\R^+)^n,$ and $\SJ\subset (\R^+)^n.$

 Since $R>\diam(T)$, for each $j\in \{1,\dots, n\}$, $\SJ\cap \be_j\R^+$ contains at least one non-zero element.
 By applying a dilation matrix $U=\text{diag}(u_1,\dots, u_n)$, we may assume that
\begin{equation}\label{eq:normal}
\be_1,\dots, \be_n\in \SJ,  \text{ and }\lambda \be_j\not\in \SJ
\text{ for all }0< \lambda <1 \text{ and all }j=1,\dots,n.
\end{equation}
Since $(T, \SJ\cap B(\bzero,R))$ also covers $(\R^+)^n\cap B(\bzero,R)$, without loss of generality, we  assume that
\begin{equation}\label{eq:R}
\SJ\subset B(\bzero,R);
\end{equation}
especially, $\SJ$ is a finite set.

The tiling of $\R^+$ has been  characterized by Odlyzko implicitly.
 Some idea of this section  comes from   Odlyzko \cite{Od78}.

  \begin{thm}\label{thm:odly} (\cite{Od78}) Let $(T,\SJ)$ be a tiling of $\R^+$. Then
  there exist a real number $c>0$ such that
$cT=E+[0,1]$, where $E\subset \Z^+$.
\end{thm}

 For $x=(x_1,\dots, x_n)\in \R^n$, we use $\|x\|_1=\max|x_j|$ to denote the $1$-norm. We say $x>0$ (or $x\geq 0$) if $x_j>0$ (or $x_j\geq 0$) for all $j=1,\dots, n$.
 %For $1\leq j\leq n$, we define $\pi_j: \R^n\to \R$ as
%\begin{equation}\label{projection}
%\pi_j(x_1,x_2,\cdots,x_n)=x_j.
%\end{equation}

\begin{lemma}\label{info-1} Let $(T,\SJ)$ be a local tiling  of $(\R^+)^n$ satisfying \eqref{eq:normal} and \eqref{eq:R}. Then
 $$\SJ \cap [0,1)^n=\{\boldsymbol{0}\} \text{ and } [0,1]^{n}\subset T.$$
Consequently, if $t, t'\in \SJ$, then $\|t-t'\|_1\geq 1$.
\end{lemma}
\begin{proof}

Let $Q$ be a $m$-face of $(\R^+)^n$ with $1\leq m\leq n$,  denote
$$D=Q\cap [0,1)^n.$$
Let $Q^\perp$ be the $(n-\dim Q)$-face of $(\R^+)^n$ complement to $Q$, that is,
$Q+ Q^\perp=(\R^+)^n$. Denote $D^\perp=Q^\perp\cap [0,1)^n$, then $D+D^\perp=[0,1)^n$.

For a $m$-face $Q$ with $1\leq m\leq n-1$, we define $\delta_Q$  as
$$
\delta_Q=\left \{
\begin{array}{l}
1,  \text{ if } \SJ\cap [0,1)^n\setminus Q=\emptyset;\\
\min\{ d(t,Q);~ t\in \SJ\cap [0,1)^n \setminus Q\},  \text{ otherwise.}
\end{array}
\right .
$$
 Set
\begin{equation}\label{eq:small}
\delta=\min\{\delta_Q;~Q \text{ is a face of } (\R^+)^n\}.
\end{equation}

We shall prove by induction on $\dim Q$ that
\begin{equation}\label{eq:Q-D}
\SJ\cap D=\{\bzero\}, \text{ and } D \subset T
\end{equation}

The result holds when $m=1$ by the assumption \eqref{eq:normal}.
Suppose  $m\geq 2$ and \eqref{eq:Q-D} holds for $m-1$.
Without loss of generality, we assume $Q$ is generated by $\be_1,\dots, \be_m$.

Denote
$$
Q_j=\text{span}(\{\be_1,\dots, \be_m\}\setminus\{\be_j\}), \quad D_j=[0,1)^n\cap Q_j.
$$
By the induction hypothesis, we have
  \begin{equation}\label{low-cube}
 D_j\subset T, \ \text{and} \
\SJ\cap D_j=\{\boldsymbol{0}\}, \ \ 1\leq j\leq m.
\end{equation}
Denote $\delta'=\delta/\sqrt{n}$  and
$$
L_j=D_j+[0,\delta')D_j^\perp.
$$
Pick $x\in L_j$ and let $T+t$ be the tile containing $x$.
Notice that $x\in [0,1)^n$ and $d(x, D_j)<\delta$.
It follows that  $t\in [0,1)^n$ and  $d(t,D_j)<\delta$ since $x-t\geq 0$.
Moreover, since $\dim D_j\leq n-1$, by the definition of $\delta$, we have $t\in D_j$, which forces $t=\bzero$
 by our induction hypothesis. This proves that
\begin{equation}\label{eq:perp}
L_j\subset T, \ j=1,\dots, m.
\end{equation}
Suppose on the contrary that there exists $t=(t_1,\dots, t_m, 0,\dots, 0)\in \SJ\cap D$ and $t\neq \bzero$, then by \eqref{low-cube},
$t_j>0$ for $j=1,\dots, m$. It is seen that
$$L_1=[0,\delta')\times [0,1)^{m-1}\times [0,\delta')^{n-m} \ \ \text{and} \ \
L_m=[0,1)^{m-1}\times[0,\delta')\times [0,\delta')^{n-m}.$$
We have that $$L_1+t=[t_1,t_1+\delta')\times \prod_{j=2}^m [t_j,t_j+1)\times [0,\delta')^{n-m}.$$
$$L_m+\be_m=[0,1)^{m-1}\times[1,1+\delta')\times [0,\delta')^{n-m}.$$
Hence $$(L_1+t)\cap (L_m+\be_m)=[t_1,\min\{1,t_1+\delta'\})\times \prod_{j=2}^{m-1} [t_j,1)\times
[1,1+\min\{t_m,\delta'\})\times [0,\delta')^{n-m},$$
which implies that $(T+t)\cap (T+\be_m)$ has positive Lebesgue measure. This contradiction proves that $\SJ\cap D=\{\bzero\}$, the first assertion of \eqref{eq:Q-D}.

%To prove the second assertion of \eqref{eq:Q-D}, we first show that $D\subset B(0,R)$.
%Suppose $D$ is not a subset of $B(0,R)$, since $R>\diam(T)$, there exists
% $x\in D\cap  B(0,R)$ and  $x\not\in T$.
%Let $T+t$ be the tile covering $x$, then $t\neq 0$ and $t\in D$, which is a contradiction.

Now $\SJ\cap D=\{\bzero\}$ implies that $T$ is the only tile intersecting $D$.
It follows that $D\cap B(\bzero,R)$ is a subset of $T$. By Lemma \ref{lem:U-T} we list below,
 $D$ is   a subset of $T$, which verifies the second assertion of  \eqref{eq:Q-D}.

Finally, set $D=[0,1)^n$ in \eqref{eq:Q-D}, we obtain  the lemma.
\end{proof}

 \begin{lemma}\label{lem:U-T}
  If $U$ is a connected set and $B(\bzero,R)\cap U\subset T$, then $U\subset T$.
  \end{lemma}

  \begin{proof}If $U$ is not a subset of $B(\bzero,R)$, then there exists $x\in U\cap B(\bzero,R)$ such that $|x|$ is as closer to $R$
 as we want, so $x\not\in T$, which is a contradiction. Therefore, we must have
  $U\subset T$.
  \end{proof}

Let `$\prec$' be the order on $(\mathbb{R}^+)^n$ defined by
$\ba \prec \bb$ if
 $\bb-\ba\geq 0$  and  $\ba\neq \bb.$

\begin{thm}\label{info-2} Let $(T,\SJ)$ be a local tiling  of $(\R^+)^n$ satisfying \eqref{eq:normal}
and \eqref{eq:R}.   Then

(i) $\SJ\subset (\mathbb{Z}^+)^n$.

(ii) There exists a subset
$E\subset (\mathbb{Z}^+)^n$ such that $T=E+[0,1]^n.$
%\noindent Consequently, $(E, \SJ)$ is a tiling of $(\Z^+)^n$.
\end{thm}

\begin{proof} %For $\bu\in \R^n$, we use $Q_\bu^\circ$ to denote the open unit cube $\bu+(0,1)^n$.
To prove the theorem, we need only prove the following two statements:
For each $\boldsymbol{z}\in (\mathbb{Z}^+)^n\cap B(\bzero,R),$
\begin{equation}\label{inter-1}
\{t\in \SJ;~t\preceq \bz\}\subset (\Z^+)^n,
\end{equation}
\begin{equation}\label{inter-2}
  \bu+(0,1)^n \text{ belongs to exactly one tile   for all } \bu\in (\Z^+)^n\cap B(\bzero,R) \text{ with }\bu\preceq \bz.
\end{equation}

In the following, we prove the above statements by induction on $\bz$.
If $\boldsymbol{z}=\boldsymbol{0}$, the statements are valid by Lemma \ref{info-1}.
Assume  $\boldsymbol{0}\prec \boldsymbol{z}$.
Suppose \eqref{inter-1} and \eqref{inter-2} hold   for
all $\bz'\in (\Z^+)^n$ with  $\boldsymbol{z}'\prec \boldsymbol{z}$, and
we  show they hold for $\boldsymbol{z}$ in the following.

First, we prove \eqref{inter-1}.
Write $\bz=(z_1,\dots, z_n)$.  By the induction hypothesis of \eqref{inter-2}, we have
$$
Y=\bigcup \{T+t;~ t\in \SJ\cap (\Z^+)^n \text{ and } t\prec \bz\}
$$
covers the set
$$
\Omega=[0,z_1+1]\times \cdots \times [0,z_n+1]\setminus (\bz+[0,1]^n),
$$
which is a rectangle missing  the `last cube' $\bz+[0,1]^n$. So if $t\in \SJ$ and $t\prec \bz$, then $t+[0,1]^n$ and $Y$ overlap, which forces $t\in (\Z^+)^n$. Therefore, no matter $\bz\in \SJ$ or not,   \eqref{inter-1} holds.

Next, we prove \eqref{inter-2}. Assume on the contrary  \eqref{inter-2} is false for $\bz$, then
the `last cube' $\bz+(0,1)^n$ intersects  at least two tiles. Denote the tiles intersecting $\bz+(0,1)^n$
by $T+t_1$, $T+t_2$, $\dots,$ $T+t_\ell$. In particular, $\bz$ does not belong to $\SJ$. Clearly
$$\{t_1,\dots, t_\ell\}\subset [0,z_1+1)\times \cdots \times [0,z_n+1).$$

Since $\Omega$ is covered by tiles $T+t$ with integral $t\prec \bz$, we conclude that $t_j$
is either integral and precedes $\bz$, or $t_j$ belongs to the `last cube' $\bz+[0,1]^n$.
%For $j\in \{1,2,\dots, \ell\}$,
%if $t_j\prec \bz$, then $t\in (\Z^+)^n$ by \eqref{inter-1} which we just proved.
In the formal case,    we must have $t_j=\bzero$,  for otherwise,
$T+t_j$ contains only a proper subset of $\bz+(0,1)^n$, so  $T$ contains only a proper subset of $(\bz-t_j)+(0,1)^n$, which contradicts our induction hypothesis on \eqref{inter-2}.
In the later case, by Lemma \ref{info-1},  at most one element of $\SJ$ belongs to $\bz+[0,1)^n$.
  It follows that there are exactly two tiles intersecting $\bz+(0,1)^n$, one
  is $T$, and the other one, which we denote by   $T+t_2$, satisfies
 $\bz\prec t_2$. So
\begin{equation}\label{dE_3}
(\bz+(0,1)^n)\cap B(\bzero,R) \subset T\cup (T+t_2).
\end{equation}
Let $\pi_j$ be the projection such that $\pi_j(x_1,\dots, x_n)=x_j$.
Notice that $\pi_j(t_2)>\pi_j(\bz)$  for at least one $j\in \{1,\dots, n\}$,  without loss of generality, let us assume that $\pi_1(t_2)>\pi_1(\bz)$.
Clearly,
$(T+t_2)\cap (\bz+(0,1)^n)= (t_2+[0,1]^n)\cap (\bz+(0,1)^n)$, so the open rectangle
$$U=\bz+(0, \pi_1(t_2)-\pi_1(\bz))\times (0,1)^{n-1},$$
as a subset of $\bz+(0,1)^n$, is not covered by $T+t_2$. Consequently, $U\cap B(\bzero,R)$ must be covered by $T$.
It follows that  $U\subset T$ by Lemma \ref{lem:U-T}.

Recall that   $\be_1\in \SJ$. Now
$(T+\be_1)\cap (T+t_2)$ contains $(\be_1+U)\cap (t_2+(0,1)^n)$ as a subset,
and the later one has positive Lebesgue measure,  which is a contradiction.
This contradiction proves that
$$(\bz+(0,1)^n)\cap B(\bzero,R)\subset T \text{ or } T+t_2.$$
In the former case, $\bz+(0,1)^n \subset T$ by Lemma \ref{lem:U-T};
in the later case, $t_2=\bz$ and clearly $\bz+(0,1)^n\subset T+\bz$.
This  verifies \eqref{inter-2} and finishes the proof of the theorem.
\end{proof}

%Therefore, we reduce the tiling problem of $(\R^+)^n$ to that of $(\Z^+)^n$.

\noindent \textbf{Proof of Theorem \ref{thm:Main-1} (ii).} It is the immediate consequence of Theorem \ref{info-2}. $\Box$

\medskip

\noindent \textbf{Proof of Corollary \ref{Main-2}.}  Let $(T, \SJ)$ be a tiling of $(\R^+)^n$ with $T=\overline{T^\circ}$. Then there is a diagonal matrix $U$ such that
$(T'=UT, \SJ'=U\SJ)$ satisfies the normalization condition \eqref{eq:normal}.
For any $R>\diam(T')$,  $(T', \SJ'\cap B(\bzero,R))$ is a local tiling of $(\R^+)^n$ satisfying
the conditions of Theorem \ref{info-2}. It follows that
$T'=E+[0,1]^n$ for some $E\subset (\Z^+)^n$, and
$\SJ'\cap B(\bzero,R)\subset (\Z^+)^n$ for all $R>0$, so $\SJ'\subset (\Z^+)^n$.

Finally, notice that  $(E+[0,1]^n, \SJ')$ is a tiling of $(\R^+)^n$ if and only if $(E, \SJ')$ is a tiling of $(\Z^+)^n$.
$\Box$

% section 9
%\input{Horn_V20}  % \label{sec:horn}
\section{\textbf{Self-affine tiles with polyhedral corners}}\label{sec:horn}
Let $T=T(\bA,\SD)$ be a $n$-dimensional self-affine  tile with expanding matrix $\bA$ and digit set $\SD$. Denote
$$
\SD_k=\SD+\bA \SD+\cdots+\bA^{k-1}\SD,
$$
then iterating $\bA T(\bA,\SD)=T(\bA, \SD)+\SD$ $k$-times, we obtain
 $$
 \bA^k T(\bA,\SD)=T(\bA,\SD)+\SD_k.
 $$

Recall that $T$ has a polyhedral corner at $x_0$ means that there is a convex polyhedral cone $C$ and a number $r>0$
 such that
 \begin{equation}\label{def:corner}
 T\cap B_n(x_0,r)=x_0+B_n(\bzero,r)\cap C.
 \end{equation}

\medskip

 \noindent \textbf{Proof of Theorem \ref{Main-3}.}   Take $k\geq 1$.
 Let $B(\bzero,\ell_k)$ be the maximal ball centered at $\bzero$ and  contained in $\bA^kB(\bzero,1)$.
Since $\bA$ is expanding, it is seen that $\ell_k\to \infty$ when $k\to \infty$.
Applying $\bA^k$ to both sides of \eqref{def:corner}, we have
$$
\bA^kT\cap \bA^kB(x_0,r)=\bA^kx_0+(\bA^k C\cap A^kB(\bzero,r)).
$$
Using $\bA^kT=T+\SD_k$, we deduce that
\begin{equation}\label{eq:9.2}
(T+\SD_k)\cap \bA^kB(x_0,r)= \bA^kx_0+(\bA^k C\cap A^kB(\bzero,r)).
\end{equation}
Notice that $B(\bA^kx_0, r\ell_k)\subset \bA^kB(x_0,r)$. Let
$$\SJ_k=\{t\in \SD_k;~  T+t \text{ intersects }\bA^kx_0+ \bA^kC\cap B(\bzero,r\ell_k-\diam(T))\}.$$
Since $T+\SD_k$ is a covering of  $\bA^kx_0+(\bA^k C\cap \bA^kB(\bzero,r))$,
$T+\SJ_k$ is a covering of
$$\bA^kx_0+ \bA^kC\cap B(\bzero,r\ell_k-\diam(T)).$$

On the other hand,
$T+\SJ_k\subset T+\SD_k=\bA^k T$, and clearly $T+\SJ_k\subset B(\bA^kx_0,r\ell_k)$; these together
with \eqref{eq:9.2} imply that
$$
T+\SJ_k\subset \bA^kT\cap B(\bA^kx_0,r\ell_k) \subset \bA^kx_0+\bA^k C,
$$
which proves that $(T, \SJ_k-\bA^kx_0)$ is a packing of $\bA^kC$.

Let $k$ be large enough so that $\ell_kr>2\diam(T)$, then $(T, {\cal J}_k-\bA^kx_0)$ is a `large' local tiling
of $\bA^kC$. Hence, by Theorem  \ref{Main-1}, $A^kC$, and also $C$,  are regular, and
  $T$ is a finite union of translations of $[0,1]^n$ up to a linear transformation.
 $\Box$

 \medskip

%\input{Reff_V8}
% References %

%\input{Appendix_A}

\begin{appendix}
\section{\textbf{Proof of Lemma \ref{lem:trapezoid}}. }

\begin{proof} For simplicity, we identify $\R^2$ to the complex plane $\mathbb C$.
Let $L_0=\{x+h \mi;~x\in [a,b]\}$ be the base line with longer line.
Assume that $t_1, t_2,\cdots, t_p$ are points in $L$ from left to right.
Suppose on the contrary that $(T, \{t_0,t_1,\dots, t_p\})$ is a tiling of $A$.
Let
$$I=\{x+y(a+h\mi); x\in [0,t_1], y\in [0,1]\}$$
be a parallelogram on the left part of $A$. Clearly $I\subset T$.
%, then $\bar{I}\subset T$ (where $\bar{I}$ denotes the closure of $I$).

Let $M$ be the largest integer such that $Mt_1<1.$
We claim that
$T\cap \left(\bigcup_{m=0}^{M-1}(I+mt_1)\right)$ is a union of translations of $I$.
To prove this, we need only prove the following two statements: for each integer $m,~0\leq m\leq M-1$, we have

$(i)$ $\SJ\cap [0, mt_1]\subset t_1\Z^+$;

$(ii)$ For every integer $0\leq u\leq m$, $I+ut_1$ belongs to one tile except a measure zero set.

We prove (i) and (ii) by induction on $m$.
Clearly (i) and (ii) holds for $m=0$.
Now we assume that (i) and (ii) holds for $m-1$ with $m\geq 1$.

First, we prove (i).
If $t\in\SJ\cap [0, mt_1)$, then $t+I$ and $\bigcup_{j=0}^{m-1}(I+jt_1)$ overlap,
%which force $t\in t_1\Z^+$, since $\bigcup_{j=0}^{m-1}(I+jt_1)$ is covered by $\{T+t;~t\in\SJ \cap [0, mt_1)\cap t_1\Z^+\} $
 we have $t\notin ((m-1)t_1,mt_1)$ by the induction hypothesis of (ii). Therefore, by the induction hypothesis of (i), no matter $mt_1\in \SJ$ or not, (i) holds for $m$.

Now we prove (ii).
Suppose on the contrary that (ii) is false. Then  $I+mt_1$ does not belong to one tile.
This first implies that $mt_1\not\in \SJ$. Secondly, if
there exists $1\leq m'\leq m-1$, such that $m't_1\in \SJ$ and
%the Lebesgue measure of
$0<\mu((T+m't_1)\cap (I+mt_1))<\mu(I)$,
%is positive and less than $\mu(I)$, then the Lebesgue measure of
then $0<\mu(T\cap(I+(m-m')t_1))<\mu(I)$,
%is positive and less than $\mu(I)$,
which contradicts the assumption (ii).
Therefore, if a tile $T+t$  satisfying that
$0<\mu((T+t)\cap (I+mt_1))<\mu(I)$,
then either $t=0$, or $mt_1<t<(m+1)t_1$. In the latter case, there is only one $t$ satisfying this property, and we denote it by $t^*$.
Then
$$I+mt_1\subset T\cup (T+t^*).$$
Denote $U=\{x+y(a+h\mi);~x\in[mt_1,t^*),y\in[0,1]\}$.
By $U\cap (T+t^*)=\emptyset$ and the above equation, we have $U\subset T$.
Then the intersection of $T+t_1$ and $T+t^*$ contains
$U+t_1$ as a subset, which is a contradiction.
So (ii) holds for $m$.

Since $t_p$ is the rightmost point of $\SJ$, $T+t_p$ must contains a relative neighborhood $B(1,r)\cap A$ of $1$, for all small enough $r(<1-Mt_1)$.
Moreover, we have
$
B(1,r)\cap A=(T+t_p)\cap B(1,r),
$
thus
\begin{equation}\label{eq:t_p}
\left(B(1,r)\cap A\right)-t_p=T\cap B(1-t_p,r).
\end{equation}
On the other hand, since $0\leq 1-t_p\leq Mt_1$ and $T\cap \left(\bigcup_{m=0}^{M-1}(I+mt_1)\right)$ is a union of translations of $I$,
then for small enough $r$, $T\cap B(1-t_p,r)$ is a half ball or a translation of $I\cap B(0,r)$,
or a translation of $I\cap B(t_1,r)$,
which contradicts with the shape of $T\cap B(1-t_p,r)$ in \eqref{eq:t_p}.
The lemma is proved.
\end{proof}
\end{appendix}

\end{document}